\documentclass[12pt,a4paper,reqno]{amsart}
\usepackage{amsmath,amsthm,amssymb,mathrsfs}
\usepackage[all,cmtip]{xy}

\numberwithin{equation}{section}

\newcounter{enumx}

\newcommand{\C}{\mathbb{C}}
\newcommand{\Q}{\mathbb{Q}}
\newcommand{\Z}{\mathbb{Z}}
\newcommand{\BF}{\mathbb{F}}
\newcommand{\BH}{\mathbb{H}}
\newcommand{\BP}{\mathbb{P}}
\newcommand{\A}{\mathcal{A}}
\newcommand{\D}{\mathcal{D}}
\newcommand{\M}{\mathcal{M}}
\newcommand{\K}{\mathcal{K}}
\newcommand{\E}{\mathcal{E}}
\newcommand{\F}{\mathcal{F}}
\newcommand{\CC}{\mathcal{C}}
\newcommand{\CR}{\mathcal{R}}
\newcommand{\G}{\mathcal{G}}
\newcommand{\I}{\mathcal{I}}
\newcommand{\W}{\mathcal{W}}

\newcommand{\CQ}{\mathcal{Q}}

\newcommand{\IC}{\mathcal{IC}}
\newcommand{\RHom}{\mathcal{RH}om}
\newcommand{\CS}{\mathscr{S}}
\newcommand{\pH}{{}^pH}
\newcommand{\ptau}{{}^p\tau}
\newcommand{\bd}{\mathbf{d}}
\newcommand{\br}{\mathbf{r}}
\newcommand{\bs}{\mathbf{s}}
\newcommand{\bbinom}[2]{\genfrac{[}{]}{0pt}{}{\,#1\,}{\,#2\,}}
\newcommand{\mof}{\text{-mof}}
\newcommand{\pmof}{\text{-pmof}}

 \DeclareMathOperator{\Ext}{Ext}
 \DeclareMathOperator{\Hom}{Hom}
 \DeclareMathOperator{\Ker}{Ker}
 
 \DeclareMathOperator{\Img}{Img}
 \DeclareMathOperator{\Id}{Id}
 \DeclareMathOperator{\Res}{Res}

\newtheorem{thm}{Theorem}[subsection]
\newtheorem{lem}[thm]{Lemma}
\newtheorem{prop}[thm]{Proposition}
\newtheorem{cor}[thm]{Corollary}

\newtheorem{prop-defn}[thm]{Proposition-Definition}
\theoremstyle{definition}
\newtheorem{defn}[thm]{Definition}
\newtheorem{rem}[thm]{Remark}
\newtheorem{exam}[thm]{Example}
\theoremstyle{remark}

\begin{document}

\title[geometric categorification]
{A geometric categorification of tensor products of
$U_q(sl_2)$-modules}
\author[H. Zheng]{Hao Zheng}
\address{Department of Mathematics, Zhongshan University, Guangzhou, Guangdong 510275, China}
\thanks{Supported in part by the National Natural Science Foundation of China (NSFC)}
\email{zhenghao@mail.sysu.edu.cn}
\date{}
\maketitle

\begin{abstract}
We give a purely geometric categorification of tensor products of
finite-dimensional simple $U_q(sl_2)$-modules and $R$-matrices on
them. The work is developed in the framework of category of perverse
sheaves and the categorification theorems are understood as
consequences of Deligne's theory of weights.
\end{abstract}

\section{Introduction}

The term categorification in mathematics refers to the process of
lifting set-theoretic concepts to the level of categories. For
example, categorification of a module $M$ over an algebra $A$ means
lifting the module $M$ to an additive or abelian category $\CC$ and,
accordingly, lifting the algebra $A$ to a collection of endofunctors
of $\CC$ and functor isomorphisms among them; the lifts are done in
such a way that the Grothendieck group of $\CC$ recovers the module
$M$ and the endofunctors and the isomorphisms among them recover the
module structure of $M$ and the algebra structure of $A$.

Categorified theories have such advantages as reflecting explicitly
the integrity and positivity of the algebraic structures involved
and, more importantly, usually providing new insights into the
background theory.

Among various known categorifications till now (cf. the review
\cite{KMS07}), algebraic approaches are playing the dominant role,
partly because there are still lacking of systematic tools for
geometric treatment. We will demonstrate here how the profound
result in modern algebraic geometry, Deligne's theory of weights
\cite{De80}, may enter to change the situation.

In the present paper, we categorify tensor products of
$U_q(sl_2)$-modules, as well as $R$-matrices on them. The former
task is accomplished in Section \ref{sec:mod:cat} by using the
decomposition theorem of Beilinson-Bernstein-Deligne-Gabber
\cite{BBD82}, which is known to be one of the remarkable
consequences of Deligne's theory of weights. As another consequence
of the weight theory, we introduce in Section \ref{sec:R:mix}
the notion of pure resolution of mixed complexes
and establish a uniqueness theorem, then use them in Section
\ref{sec:R:R} and \ref{sec:R:cat} to categorify $R$-matrices. Thanks
to these powerful tools, our categorification is able to be
fulfilled in a very simple and elegant way.

The main part of the paper consists of Section \ref{sec:mod} and
Section \ref{sec:R}. Further remarks on the motivations and
expositions of this work will be given in the beginning of them.

\tableofcontents

\section{Preliminaries}

The references for Section \ref{sec:pre:u} are \cite{Kas95},
\cite{Lu93} and the references for Section \ref{sec:pre:perverse}
are \cite{BBD82}, \cite{Bor84}, \cite{KS90}.

\subsection{The algebra $U_\A$}\label{sec:pre:u}

Throughout this paper, $\A=\Z[q,q^{-1}]$ denotes the Laurent
polynomial ring and we set
$$[n]_q=\frac{q^n-q^{-n}}{q-q^{-1}}, \quad
  [n]_q!=[1]_q[2]_q\cdots[n]_q, \quad
  \bbinom{n}{r}_q=\prod_{t=1}^r\frac{[n-r+t]_q}{[t]_q}.
$$

The quantum enveloping algebra $U=U_q(sl_2)$ is the $\Q(q)$-algebra
defined by the generators $K, K^{-1}, E, F$ and the relations
\begin{equation}\label{eqn:r1}
\begin{split}
  & KK^{-1}=K^{-1}K=1, \\
  & KE=q^2EK, \quad KF=q^{-2}FK, \\
  & EF-FE=\frac{K-K^{-1}}{q-q^{-1}}.
\end{split}
\end{equation}
It is a Hopf algebra with comultiplication
\begin{equation}\label{eqn:comul}
\begin{split}
  & \Delta K = K \otimes K, \\
  & \Delta E = E \otimes 1 + K \otimes E, \\
  & \Delta F = F \otimes K^{-1} + 1 \otimes F,
\end{split}
\end{equation}
counit
\begin{equation}
  \varepsilon(K) = 1, \quad
  \varepsilon(E) = \varepsilon(F) = 0,
\end{equation}
and antipode
\begin{equation}
  S(K) = K^{-1}, \quad
  S(E) = -K^{-1}E, \quad
  S(F) = -FK.
\end{equation}

To emphasize the integrity of the finite-dimensional representations
of $U$, we work on an alternative algebra $U_\A$ which is defined as
the $\A$-subalgebra of $U$ generated by $K, K^{-1}, E^{(n)},
F^{(n)}$, $n\ge0$ where
\begin{equation}
  \label{eqn:r2}
  E^{(n)}=\frac{E^n}{[n]_q!}, \quad F^{(n)}=\frac{F^n}{[n]_q!}.
\end{equation}

\smallskip

For every integer $d\ge0$, there is a simple $U_\A$-module
\begin{equation}
  \Lambda_d=U_\A/(U_\A\cap\I_d)
\end{equation}
where $\I_d$ is the left ideal of $U$ generated by $E, K-q^d$ and
$F^{d+1}$. Tensoring with $\Q(q)$, they recover the
finite-dimensional simple $U$-modules. The elements
\begin{equation}\label{eqn:bld}
  v_r=\bar{F}^{(r)}, \quad r=0,1,\dots,d
\end{equation}
form a basis of $\Lambda_d$ and (we define $v_{-1}=v_{d+1}=0$)
\begin{equation}\label{eqn:mld}
\begin{split}
  & Kv_r = q^{d-2r}v_r, \\
  & Ev_r = [d-r+1]_qv_{r-1}, \\
  & Fv_r = [r+1]_qv_{r+1}.
\end{split}
\end{equation}

More generally, for a composition $\bd=(d_1,d_2,\dots,d_l)$ of $d$
(i.e. a sequence of nonnegative integers summing up to $d$), let
\begin{equation}
  \Lambda_\bd = \Lambda_{d_1} \otimes \Lambda_{d_1} \otimes \cdots \Lambda_{d_l}
\end{equation}
be the tensor product of $U_\A$-modules. It has a standard basis
\begin{equation}\label{eqn:vr}
  v_\br = v_{r_1} \otimes v_{r_2} \otimes \cdots \otimes v_{r_l}
\end{equation}
with $\br=(r_1,r_2,\dots,r_l)$ running over the compositions
satisfying $r_k\le d_k$, $k=1,2,\dots,l$.

\smallskip

Let $\varrho: U_\A\to U_\A^{op}$ be the $\A$-algebra isomorphism
defined on the generators by
\begin{equation}
  \varrho(K)=K, \quad \varrho(E)=qKF, \quad \varrho(F)=qK^{-1}E.
\end{equation}
By an inner product of a $U_\A$-module $M$ we mean a non-degenerate
symmetric bilinear form
\begin{equation*}
  (,): M\times M \to \A
\end{equation*}
satisfying
\begin{equation}
  (xu,w)=(u,\varrho(x)w) \quad \text{for $x\in U_\A$, $u,w\in M$}.
\end{equation}

Since $\varrho$ is compatible with the comultiplication of $U_\A$:
$$(\varrho\otimes\varrho)\Delta(x)=\Delta\varrho(x) \quad \text{for $x\in U_\A$},$$
inner products of $U_\A$-modules $M_1,M_2$ automatically give rise
to an inner product of the tensor product module $M_1 \otimes M_2$
such that
\begin{equation*}
  (u_1\otimes u_2,w_1\otimes w_2) = (u_1,w_1)(u_2,w_2) \quad
  \text{for $u_1,w_1\in M_1, u_2,w_2\in M_2$}.
\end{equation*}

The simple $U_\A$-module $\Lambda_d$ has a unique inner product up
to a constant, which we will normalize as
\begin{equation}\label{eqn:pld}
  (v_r,v_{r'})=\delta_{rr'}\bbinom{d}{r}_qq^{-r(d-r)}.
\end{equation}
They automatically extend to inner products of the tensor product
modules $\Lambda_\bd$.

\subsection{Perverse sheaves}\label{sec:pre:perverse}

Let $X$ be a complex algebraic variety. We denote by
$\D(X)=\D^b_c(X)$ the bounded derived category of constructible
$\C$-sheaves on $X$ and denote by $\M(X)$ the full subcategory
consisting of perverse sheaves. An object of $\D(X)$ is also
referred to as a complex. Given a connected algebraic group $G$
acting on $X$, let $\M_G(X)$ denote the full subcategory of $\M(X)$
whose objects are the $G$-equivariant perverse sheaves on $X$.

We denote by $D: \D(X)\to\D(X)^\circ$ the Verdier duality functor.
For an integer $n$, let $[n]: \D(X)\to\D(X)$ denote the shift
functor and let $\pH^n: \D(X)\to\M(X)$ denote the $n$-th perverse
cohomology functor. There are functor isomorphisms
$$D^2=\Id, \quad \pH^n[j] = \pH^{n+j}, \quad D[n]=[-n]D.$$

A complex $C\in\D(X)$ is said to be semisimple if
$C\cong\oplus_n\pH^n(C)[-n]$ and if $\pH^n(C)\in\M(X)$ is semisimple
for all $n$. A semisimple complex $C\in\D(X)$ is called
$G$-equivariant if $\pH^n(C)\in\M_G(X)$ for all $n$.

\smallskip

The Ext groups of $C,C'\in\D(X)$ are the $\C$-linear spaces
$$\Ext^n_{\D(X)}(C,C') = \Hom_{\D(X)}(C,C'[n]) = \BH^n\RHom(C,C');$$
they satisfy
\begin{enumerate}
\setlength{\itemsep}{.3ex}
\item
$\Ext^\bullet_{\D(X)} (C[n],C'[n']) = \Ext^{\bullet-n+n'}_{\D(X)}
(C,C')$.

\item
$\Ext^\bullet_{\D(X)} (C,DC') = \Ext^\bullet_{\D(X)} (C',DC) =
\BH^\bullet D(C\otimes C')$.

\item
$\Ext^n_{\D(X)}(C,C')=0$ for $C,C'\in\M(X)$ and $n<0$.

\item
For simple perverse sheaves $C,C'\in\M(X)$, $\Ext^0_{\D(X)}(C,C')$
is isomorphic to $\C$ if $C\cong C'$ and vanishes otherwise.

\setcounter{enumx}{\value{enumi}}
\end{enumerate}

Let $f: X \to Y$ be a morphism of algebraic varieties. There are
induced functors $f_!,f_*: \D(X)\to\D(Y)$ and $f^!,f^*:
\D(Y)\to\D(X)$. For reader's convenience we list some properties of
these functors as follows.
\begin{enumerate}
\setlength{\itemsep}{.3ex} \setcounter{enumi}{\value{enumx}}
\item
$Df^*=f^!D$ and $Df_!=f_*D$.

\item
$f^*$ is left adjoint to $f_*$ and $f_!$ is left adjoint to $f^!$.

\item
If $f$ is proper, then $f_!=f_*$.

\item
If $f$ is smooth with connected nonempty fibers of dimension $d$,
then $f^*[d]=f^![-d]$ which induces a fully faithful functor $\M(Y)
\to \M(X)$ and sends simple perverse sheaves to simple perverse
sheaves.

\item
There are natural isomorphisms for $C,C'\in\D(Y)$, $C''\in\D(X)$
\begin{align*}
  & f^*(C\otimes C') = f^*C\otimes f^*C', \\
  & f^!\RHom(C,C') = \RHom(f^*C,f^!C'), \\
  & f_!C''\otimes C = f_!(C''\otimes f^*C), \\
  & \RHom(f_!C'',C) = f_*\RHom(C'',f^!C).
\end{align*}
In particular,
\begin{align*}
  & f_!f^*C = f_!(\C_X\otimes f^*C) = f_!\C_X\otimes C, \\
  & f_*f^!C = f_*\RHom(\C_X,f^!C) = \RHom(f_!\C_X,C).
\end{align*}

\item
Assume $f: X\to Y$ is a $G$-equivariant morphism. If $C\in\M_G(X)$,
then $\pH^n(f_!C)\in\M_G(Y)$ for all $n$. If $C'\in\M_G(Y)$, then
$\pH^n(f^*C')\in\M_G(X)$ for all $n$.

\item
Assume $f: X\to Y$ is a (locally trivial) principal $G$-bundle. The
functor $f^*[\dim G]$ and the functor $f_\flat=\pH^{-\dim G}f_*$
define an equivalence of the categories $\M_G(X)$, $\M(Y)$.

\item
$(fg)^*=g^*f^*$ and $(fg)_!=f_!g_!$ for morphisms $f: X\to Y$, $g:
Y\to Z$.

\item
(Proper base change) $f^*g_!=g'_!f'^*$ holds for the cartesian
square
$$\xymatrix{
  X\times_YY' \ar[r]^-{g'} \ar[d]_{f'} & X \ar[d]^f \\
  Y' \ar[r]_g & Y
}
$$

\setcounter{enumx}{\value{enumi}}
\end{enumerate}

For a subvariety $S\subset X$ and a complex $C\in\D(X)$ we also
write $C|_S$ instead of $j_S^*C$ where $j_S: S\to X$ is the
inclusion.

For a locally closed irreducible smooth subvariety $S \subset X$, we
denote by $\IC(\overline{S}) \in \M(X)$ the simple perverse sheaf
(the intersection complex) defined as the intermediate extension of
the shifted constant sheaf $\C_S[\dim S]$. Below is a rather deep
result on the interplay between proper morphisms and perverse
sheaves.

\begin{enumerate}
\setlength{\itemsep}{.3ex} \setcounter{enumi}{\value{enumx}}
\item
(Decomposition theorem) If $f: X\to Y$ is a proper morphism, then
for every locally closed irreducible smooth subvariety $S \subset
X$, $f_!\IC(\overline{S}) \in \D(Y)$ is a semisimple complex.

\setcounter{enumx}{\value{enumi}}
\end{enumerate}

The following implications of the decomposition theorem will be used
in this paper.

\begin{enumerate}
\setlength{\itemsep}{.3ex} \setcounter{enumi}{\value{enumx}}
\item
If $f: X\to Y$ is a proper morphism with $X$ smooth, then
$f_!\C_X\in \D(Y)$ is a semisimple complex.

\item
Assume a connected algebraic group $G$ acts on a variety $X$, having
finitely many orbits. Then the $G$-equivariant simple perverse
sheaves on $X$ are exactly those $\IC(\overline{S})$ for various
$G$-orbits $S$. Therefore, by the decomposition theorem, if $f: X\to
Y$ is a proper morphism then $f_!$ sends $G$-equivariant semisimple
complexes to semisimple complexes.
\end{enumerate}

\subsection{Partial flag varieties}\label{sec:pre:flag}

Let $G \supset P \supset B$ be a connected reductive algebraic
group, a parabolic subgroup and a Borel subgroup of it,
respectively. We have a partial flag variety $X=G/P$. Let
$\W=N_G(T)/T$ be the Weyl group with respect to a fixed maximal
torus $T\subset B$, and for every element $w\in\W$ we fix a
representative $\dot{w}\in N_G(T)$.

We denote by $\W_P\subset\W$ the subgroup corresponding to $P$ and
denote by $\W^P$ the set of shortest representatives of the cosets
$\W/\W_P$. The $B$-orbits partition $X$ into a finite number of
affine cells (Bruhat decomposition)
\begin{equation}
  X=\bigsqcup_{w\in\W^P} X_w
\end{equation}
where $X_w = B\dot{w}P/P$. The subvarieties $X_w$ are referred to as
Schubert cells, and their closures are called Schubert varieties.

It follows that, up to isomorphism, the $B$-equivariant semisimple
complexes on $X$ are finite direct sums of $\IC(\overline{S})[j]$
for various Schubert cells $S$ and integers $j$.

\smallskip

The main concern of this paper is the case that $G=GL(W)$ is a
general linear group, where $W$ is a complex linear space of
dimension $d$, and that $B$ is the Borel subgroup preserving a fixed
complete flag
$$0=W_0\subset W_1\subset W_2\subset\cdots\subset W_d=W.$$
Given an ascending sequence of integers $0\le r_1\le r_2\le
\cdots\le r_n\le d$, there is a partial flag variety
\begin{equation*}
  X = G/P = \{ 0\subset V_1\subset V_2\subset\cdots\subset V_n\subset W \mid
  \dim V_i=r_i, \; i=1,2,\dots,n \},
\end{equation*}
where $P$ is the parabolic subgroup preserving the subspaces
$W_{r_i}$, $i=1,2,\dots,n$.

\smallskip

The following lemma is proved by Kazhdan-Lusztig \cite{KL80}.

\begin{lem}\label{lem:flag2}
For each pair of Schubert cells $X_w,X_v$ of $X$,
$\pH^n(\IC(\overline{X}_w)|_{X_v})=0$ unless $n \equiv \dim X_w+\dim
X_v \pmod{2}$.
\end{lem}

The proof of the following lemma is borrowed from \cite[3.4]{BGS96}.

\begin{lem}\label{lem:flag3}
Let $S\subset X$ be a subvariety consisting of Schubert cells
$X_1\sqcup X_2\sqcup\cdots\sqcup X_k$. Then, for $B$-equivariant
semisimple complexes $C,C'\in\D(X)$, we have
\begin{equation*}
  \Ext^\bullet_{\D(S)}(C|_S,D(C'|_S)) \cong
  \oplus_{i=1}^k \Ext^\bullet_{\D(X_i)}(C|_{X_i},D(C'|_{X_i})).
\end{equation*}
\end{lem}

\begin{proof}
We may assume $C=\IC(\overline{X}_w)$, $C'=\IC(\overline{X}_v)$
where $X_w,X_v$ are Schubert cells. Setting
$$S_p=\sqcup_{\dim X_i=\dim S-p}X_i,$$
we get a filtration of closed subvarieties
$$S\supset S\setminus S_0\supset S\setminus (S_0\sqcup S_1)\supset\cdots\supset\emptyset.$$
Then $\Ext^\bullet_{\D(S)}(C|_S,D(C'|_S)) = \BH^\bullet D(C\otimes
C'|_S)$ is the limit of a spectral sequence with $E_1$-term
\begin{align*}
  E_1^{p,q} & = \BH^{p+q} D(C\otimes C'|_{S_p})
  = \oplus_{\dim X_i=\dim S-p} \BH^{p+q} D(\IC(\overline{X}_w) \otimes \IC(\overline{X}_v)|_{X_i}).
\end{align*}
Since both $\IC(\overline{X}_w)|_{X_i},\IC(\overline{X}_v)|_{X_i}$
are direct sums of shifted constant sheaves, by Lemma
\ref{lem:flag2} $\BH^n D(\IC(\overline{X}_w) \otimes
\IC(\overline{X}_v)|_{X_i}) = 0$ unless $n \equiv \dim X_w+\dim X_v
\pmod{2}$. The $E_1$-term therefore ``vanishes like a chess-board''.
It follows that the spectral sequence degenerates at the $E_1$-term
and we deduce that
$$\Ext^n_{\D(S)}(C|_S,D(C'|_S))
  \cong \oplus_{p+q=n} E_1^{p,q}
  = \oplus_{i=1}^k \Ext^n_{\D(X_i)}(C|_{X_i},D(C'|_{X_i})).
  \eqno\qed
$$
\renewcommand{\qed}{}
\end{proof}

\begin{cor}\label{cor:flag3}
Let $S = S_1\sqcup S_2\sqcup\cdots\sqcup S_k \subset X$ be a
subvariety with each $S_i$ being a union of Schubert cells. Then,
for $B$-equivariant semisimple complexes $C,C'\in\D(X)$, we have
\begin{equation*}
  \Ext^\bullet_{\D(S)}(C|_S,D(C'|_S)) \cong
  \oplus_{i=1}^k \Ext^\bullet_{\D(S_i)}(C|_{S_i},D(C'|_{S_i})).
\end{equation*}
\end{cor}

\smallskip

Next, we recall the Bott-Samelson resolution of a given Schubert
variety $\overline{X}_w$. Fix a reduced word $w=s_{i_1}s_{i_2}\cdots
s_{i_t}$ and set
$$Z = \overline{B\dot{s}_{i_1}B} \times \overline{B\dot{s}_{i_2}B} \times \cdots \times \overline{B\dot{s}_{i_t}B} / B^t$$
in which $B^t$ acts by the equation
$$(x_1,x_2,\dots,x_t) \cdot (g_1,g_2,\dots,g_t) = (x_1g_1,g_1^{-1}x_2g_2,\dots,g_{t-1}^{-1}x_tg_t).$$
The projection
$$\pi: Z\to\overline{X}_w, \quad [x_1,x_2,\dots,x_t]\mapsto [x_1x_2\cdots x_t]$$
then gives rise to a resolution of singularities. Indeed, $Z$ is an
iterated $\BP^1$-bundle.

The following variation of Bott-Samelson resolution for
Grassmannians will be used in Section \ref{sec:mod:res}.

\begin{lem}\label{lem:desingular}
Given a Grassmannian variety
$$X = GL(W)/P = \{V\subset W \mid \dim V=r\},$$
for each Schubert variety $\overline{X}_w \subset X$ and for each
integer $0\le d'\le d$, there exists a resolution of singularities
$\pi: Z\to \overline{X}_w$ such that the preimage of each
$$Y_{r'}=\{ V\in \overline{X}_w \mid \dim(V\cap W_{d'})=r' \}$$
is a smooth subvariety of $Z$.
\end{lem}

\begin{proof}
Choose a decomposition $w=s_{i_1}s_{i_2}\cdots s_{i_t}w'$ such that
$\ell(w)=t+\ell(w')$, $s_{i_j}$ are simple reflections satisfying
$\dot{s}_{i_j}W_{d'} = W_{d'}$ and $\ell(w')$ is minimal in
possible. It is straightforward to check that
$$\overline{X}_{w'} = \overline{B\dot{w}'P/P}$$
is a Grassmannian variety and has each
$$Y'_{r'} = \{ V\in \overline{X}_{w'} \mid \dim(V\cap W_{d'})=r' \}$$
as a smooth subvariety. Following the spirit of Bott-Samelson
resolution, we can form a resolution of singularities $\pi: Z\to
\overline{X}_w$ with
$$Z = \overline{B\dot{s}_{i_1}B} \times \overline{B\dot{s}_{i_2}B} \times \cdots \times
  \overline{B\dot{s}_{i_t}B} \times \overline{B\dot{w}'P} / B^t\times P.
$$
Moreover, each
$$\pi^{-1}(Y_{r'}) = \{ [x_1,x_2,\dots,x_t,x'] \in Z \mid \dim(x'W_r \cap W_{d'})=r' \}$$
is a $Y'_{r'}$-bundle over an iterated $\BP^1$-bundle, hence is a
smooth subvariety of $Z$. This completes the proof.
\end{proof}

\section{Categorification of $U_\A$-modules}\label{sec:mod}

Categorification of representations of quantum groups is a fairly
new topic. For the simplest cases, the irreducible representations
of $U_q(sl_2)$ and the tensor products of the fundamental
representation of $U_q(sl_2)$, the picture has been fairly clear;
categorifications are implemented via both algebraic and geometric
approaches (cf. for instance \cite{BFK99}, \cite{CR04}).

The next development along this direction is the very recent work by
Frenkel-Khovanov-Stroppel \cite{FKS05}, in which the authors
succeeded in categorifying tensor products of general
$U_q(sl_2)$-modules. The work used at full length many deep results
on representations of Lie algebras.

\smallskip

In this section, we do the same as \cite{FKS05}, but in quite a
different way. The categorification is fulfilled in the framework of
perverse sheaves on Grassmannians. Moreover, it is tailored to set
up an initial stage for the categorification of representations of
general quantum groups via the geometry of Nakajima's quiver
varieties \cite{Na94}.

Quiver varieties are very natural and successful tools in the study
of representations of Kac-Moody algebras \cite{Na01} and their
quantum analogues \cite{Lu91}, \cite{KSa97}. Naturally the same is
expected for categorification. As will be justified below,
microlocal perverse sheaves \cite{KS90} \cite{Wa04} \cite{GMV05} on
them (rather than homology groups or perverse sheaves as usually
treated) turn out to provide the appropriate setting for this goal.

Notice that the tensor product varieties \cite{Na01} \cite{Ma03}
associated to tensor products of $U_q(sl_2)$-modules are conic
Lagrangian subvarieties of the cotangent bundles of Grassmiannians.
A standard result then states the categories of perverse sheaves we
use in this section are equivalent via microlocalization functor to
the categories of microlocal perverse sheaves supported on these
varieties. That being said, our categorification is a priori able to
be achieved alternatively in the framework of microlocal perverse
sheaves.

Further examination by examples reveals that microlocal perverse
sheaves on Nakajima's quiver varieties do carry the right
information necessary for extending the present work to general
quantum groups. Actually, it was this observation that motivated the
present paper.

Nevertheless, carrying the full plan out needs, that will be our
next concern, substantial developments on many aspects of the theory
of microlocal perverse sheaves. A version for schemes, which is
still vacant from the literature, is especially welcome.

\smallskip

The section is organized as follows. We establish the
categorification theorem in the first three subsections. The
construction is straightforward and elementary. The only nontrivial
tool used is the decomposition theorem.

In Section \ref{sec:mod:ip} and \ref{sec:mod:d}, we realize inner
product of $U_\A$-modules and the bar involution of $U_\A$ via
certain functors.

In Section \ref{sec:mod:abel}, we translate the categorification
into an abelian version. This abelian version exhibits many
resemblances with the work \cite{FKS05}, but at this moment we have
no proof to their equivalence.

The last three subsections are devoted to identify the standard
$U_\A$-modules with what we have categorified. The task can be done
in more elementary ways, but we stick to our treatment because of
its advantage of being less dependent on the algebraic knowledge of
$U_\A$-modules. This is important when we are confronting with other
quantum groups.

\smallskip

Notice the resemblance of this work with Lusztig's treatment
\cite{Lu91}\cite{Lu93} for canonical basis of quantum groups, for
example, canonical bases being constructed explicitly from simple
perverse sheaves, and the usage of inner product.

\subsection{The category $\CQ_\bd$}\label{sec:mod:Q}

Let $W$ be a complex linear space of dimension $d$ and fix a
complete flag
\begin{equation}\label{eqn:Q1}
  0=W_0\subset W_1\subset W_2\subset\cdots\subset W_d=W.
\end{equation}
We have for each integer $0\le r\le d$ a Grassmannian variety
\begin{equation}
  X_d^r = \{ V\subset W \mid \dim V=r \}.
\end{equation}
It is convenient to set $X_d^r=\emptyset$ for $r<0$ or $r>d$.

Given a composition $\bd=(d_1,d_2,\dots,d_l)$ of $d$, let $P_\bd$
denote the parabolic subgroup of $G=GL(W)$ preserving the subspaces
$W_{d_1+d_2+\cdots+d_k}$, $k=1,2,\dots,l$. We denote the set of the
$P_\bd$-orbits of $X_d^r$ as $\CS_\bd^r$. It is indexed by the
compositions $\br=(r_1,r_2,\dots,r_l)$ of $r$ satisfying $r_k\le
d_k$, $k=1,2,\dots,l$; associated to $\br=(r_1,r_2,\dots,r_l)$ is
the orbit
\begin{equation}
  X_\br = \{ V\in X_d^r \mid \dim(V\cap W_{d_1+\dots+d_k})=r_1+\dots+r_k , \; k=1,2,\dots,l \}.
\end{equation}

In the specific case $\bd = (1,1,\dots,1)$, $P_\bd$ is the Borel
subgroup $B\subset G$ preserving the complete flag \eqref{eqn:Q1},
hence the $P_\bd$-orbits coincide with the Schubert cells of
$X_d^r$.

\begin{rem}
The conormal variety to the $P_\bd$-orbits of $\sqcup_rX_d^r$ is
precisely the tensor product varieties \cite{Na01} \cite{Ma03}
associated to the $U$-module $\Q(q)\times_\A\Lambda_\bd$.
\end{rem}

Let $\CQ_\bd^r$ be the full subcategory of $\D(X_d^r)$ consisting of
the $P_\bd$-equivariant semisimple complexes. Up to isomorphism the
objects from $\CQ_\bd^r$ are finite direct sums of
$\IC(\overline{X}_\br)[j]$ for various $P_\bd$-orbits $X_\br$ and
integers $j$.

The categories $\CQ_\bd^r$ are additive (but neither abelian nor
triangulated in general). We set
\begin{equation}
  \CQ_\bd = \bigoplus_r \CQ_\bd^r.
\end{equation}
By the Grothendieck group $Q_\bd$ of $\CQ_\bd$ we mean the free
$\A$-module defined by the generators each for an isomorphism class
of objects from $\CQ_\bd$ and the relations
\begin{enumerate}
\renewcommand{\theenumi}{\roman{enumi}}
\setlength{\itemsep}{.3ex}
  \item $[C\oplus C']=[C]+[C']$, for $C,C'\in\CQ_\bd$;
  \item $[C[1]]=q^{-1}[C]$, for $C\in\CQ_\bd$.
\end{enumerate}
It has a canonical basis
\begin{equation}\label{eqn:cb}
  b_\br = [\IC(\overline{X}_\br)], \quad X_\br\in\sqcup_r\CS_\bd^r.
\end{equation}

\begin{exam}
For $\bd=(d)$ and $0\le r\le d$, $X_d^r$ consists of a single
$P_\bd$-orbit: $\CS_\bd^r = \{ X_{(r)}=X_d^r \}$. Therefore, $Q_\bd$
has a canonical basis
\begin{equation}
  b_{(r)}=[\IC(X_d^r)], \quad 0\le r\le d.
\end{equation}
Thus $Q_\bd \cong \oplus_{0\le r\le d} \A$.
\end{exam}

\begin{exam}
For $\bd=(2,2)$ and $r=2$, we have $\CS_\bd^r = \{ X_{(2,0)},
X_{(1,1)}, X_{(0,2)} \}$ where $X_{(2,0)}$ is a point, $X_{(1,1)}$
consists of four Schubert cells and $X_{(0,2)}\cong\C^4$ is the top
Schubert cell. We have met a singular Schubert variety
$\overline{X}_{(1,1)} = X_{(1,1)} \sqcup X_{(2,0)}$.
\end{exam}

\subsection{The functors $\K, \E^{(n)}, \F^{(n)}$}\label{sec:mod:kef}

For every integer $n\ge0$, we have a diagram
\begin{equation}\label{eqn:ef}
\xymatrix@1{
  X_d^r & \ar[l]_{p} X_d^{r,r+n} \ar[r]^-----{p'} & X_d^{r+n}
}
\end{equation}
where
\begin{equation}
  X_d^{r,r+n} = \{ V\subset V'\subset W \mid \dim V=r, \; \dim V'=r+n \},
\end{equation}
and $p(V,V')=V$, $p'(V,V')=V'$. Note that $p$ (resp. $p'$) is an
$X_{r+n}^n$-bundle (resp. $X_{d-r}^n$-bundle) and that
\begin{align*}
  & \dim X_d^{r,r+n} - \dim X_d^{r+n} = nr, \\
  & \dim X_d^{r,r+n} - \dim X_d^r = n(d-n-r).
\end{align*}

Define functors
\begin{equation}
\begin{array}{lcl}
  \K_{r} = [2r-d] & : & \D(X_d^r) \to \D(X_d^r), \\
  \E^{(n)}_{r+n} = p_!p'^*[nr] & : & \D(X_d^{r+n}) \to \D(X_d^r), \\
  \F^{(n)}_{r} = p'_!p^*[n(d-n-r)] & : & \D(X_d^r) \to \D(X_d^{r+n}),
\end{array}
\end{equation}
and assemble them into endofunctors of $\oplus_r \D(X_d^r)$
\begin{equation}
  \K = \oplus_r \K_r, \quad
  \E^{(n)} = \oplus_r \E^{(n)}_r, \quad
  \F^{(n)} = \oplus_r \F^{(n)}_r.
\end{equation}
We also abbreviate $\E_r^{(1)}, \F_r^{(1)}, \E^{(1)}, \F^{(1)}$ to
$\E_r, \F_r, \E, \F$, respectively.

\smallskip

The following proposition states that these functors induce
endofunctors of $\CQ_\bd$.

\begin{prop}
We have $\K_rC \in \CQ_\bd^r$, $\E^{(n)}_rC \in \CQ_\bd^{r-n}$ and
$\F^{(n)}_rC \in \CQ_\bd^{r+n}$ for $C \in \CQ_\bd^r$.
\end{prop}

\begin{proof}
The statement for $\K$ is trivial. We prove the proposition for $\F$
and a similar argument applies to $\E$. Since $p$ is a Grassmannian
bundle and is $P_\bd$-equivariant, $p^*C$ is a $P_\bd$-equivariant
semisimple complex. Since $p'$ is proper and is also
$P_\bd$-equivariant, by the decomposition theorem $p'_!p^*C$, and
therefore $\F^{(n)}_rC$, is a $P_\bd$-equivariant semisimple
complex.
\end{proof}

\begin{exam}\label{exam:d1}
For $\bd=(d)$, we have
$$\F^{(r)}\IC(X_d^0) = \IC(X_d^r)$$
and
\begin{align*}
  & \K\IC(X_d^r) = \IC(X_d^r)[2r-d], \\
  & \E\IC(X_d^r) = \oplus_{j=0}^{d-r} \IC(X_d^{r-1})[d-r-2j], \\
  & \F\IC(X_d^r) = \oplus_{j=0}^r \IC(X_d^{r+1})[r-2j],
\end{align*}
in agreement with \eqref{eqn:bld}, \eqref{eqn:mld}.
\end{exam}

\subsection{Categorification theorem}\label{sec:mod:cat}

We shall show that the endofunctors $\K$, $\E^{(n)}$, $\F^{(n)}$
categorify the generators $K, E^{(n)}, F^{(n)}$ of $U_\A$. Our first
two propositions are obvious.

\begin{prop}
The functor $\K$ is an autoequivalence.
\end{prop}

\begin{prop}
We have functor isomorphisms
\begin{align*}
  & \K\E = \E\K[-2], \quad \K\F = \F\K[2].
\end{align*}
\end{prop}

\begin{prop}
We have functor isomorphisms
\begin{align*}
  & \E^{(n-1)} \E \cong \bigoplus\limits_{j=0}^{n-1} \E^{(n)} [n-1-2j], \\
  & \F^{(n-1)} \F \cong \bigoplus\limits_{j=0}^{n-1} \F^{(n)} [n-1-2j].
\end{align*}
\end{prop}

\begin{proof}
We prove the second isomorphism. Consider the commutative diagram
$$\xymatrix{
  X^{r,r+n}_d \ar[rr]^{p'_3} \ar[dd]_{p_3} & & X^{r+n}_d \\
  & Y \ar[lu]_{p_{13}} \ar[r]^{p_{23}} \ar[d]^{p_{12}} & X^{r+1,r+n}_d \ar[u]_{p'_2} \ar[d]^{p_2} \\
  X^r_d & \ar[l]_{p_1} X^{r,r+1}_d \ar[r]^{p'_1} & X^{r+1}_d
}
$$
where
\begin{equation*}
  Y = \{ V_1\subset V_2\subset V_3\subset W \mid \dim V_1=r, \; \dim V_2=r+1, \; \dim V_3=r+n \},
\end{equation*}
$p_{ij}(V_1,V_2,V_3)=(V_i,V_j)$ and $p_i,p'_i$, $i=1,2,3$ are given
as \eqref{eqn:ef}. We have
\begin{align*}
  \F^{(n-1)}_{r+1} \F_r
  & = (p'_2)_! (p_2)^* (p'_1)_! (p_1)^*[k]
  = (p'_2)_! (p_{23})_! (p_{12})^* (p_1)^*[k] \\
  & \quad \quad \quad \quad = (p'_3)_! (p_{13})_! (p_{13})^* (p_3)^*[k]
\end{align*}
in which we set
$$k=(d-r-1)+(n-1)(d-n-r)$$
and the second equality is by proper base change. Since $p_{13}$ is
a $\BP^{n-1}$-bundle,
$$(p_{13})_!\C_Y \cong \bigoplus_{j=0}^{n-1}\C_{X^{r,r+n}_d}[-2j]$$
by the decomposition theorem. Therefore,
$$(p_{13})_! (p_{13})^* = (p_{13})_!\C_Y\otimes- \cong \bigoplus_{j=0}^{n-1}[-2j].$$
It follows that
\begin{equation*}
  \F^{(n-1)}_{r+1} \F_r
  \cong \bigoplus\limits_{j=0}^{n-1} (p'_3)_! (p_3)^* [k-2j]
  = \bigoplus\limits_{j=0}^{n-1} \F^{(n)}_{r} [n-1-2j].
\end{equation*}
Assembling the isomorphism for various $r$, we prove the
proposition.
\end{proof}

\begin{prop}
There is a functor isomorphism
\begin{align*}
  & \E_{r+1}\F_r \oplus \bigoplus\limits_{0\le j<2r-d} \Id[(2r-d)-1-2j] \\
  & \cong \F_{r-1}\E_r \oplus \bigoplus\limits_{0\le j<d-2r} \Id[(d-2r)-1-2j].
\end{align*}
\end{prop}

\begin{proof}
We start with the commutative diagrams
$$\xymatrix{
  Y \ar[rr]^{p'_3} \ar[dd]_{p_3} & & X^r_d \\
  & Y' \ar[lu]_{p'_{12}} \ar[r]^{p'_{23}} \ar[d]^{p'_{13}} & X^{r,r+1}_d \ar[u]_{p_1} \ar[d]^{p'_1} \\
  X^r_d & \ar[l]_{p_1} X^{r,r+1}_d \ar[r]^{p'_1} & X^{r+1}_d
} \quad \quad \xymatrix{
  Y \ar[rr]^{p'_3} \ar[dd]_{p_3} & & X^r_d \\
  & Y'' \ar[lu]_{p''_{23}} \ar[r]^{p''_{13}} \ar[d]^{p''_{12}} & X^{r-1,r}_d \ar[u]_{p'_2} \ar[d]^{p_2} \\
  X^r_d & \ar[l]_{p'_2} X^{r-1,r}_d \ar[r]^{p_2} & X^{r-1}_d
}
$$
where
\begin{align*}
  Y & = \{ V,V'\in X_d^r \mid \dim(V+V') \le r+1 \} \\
  & = \{ V,V'\in X_d^r \mid \dim(V\cap V') \ge r-1 \}, \\
  Y' & = \{ V_1,V_2\subset V_3\subset W \mid \dim V_1=\dim V_2=r, \; \dim V_3=r+1 \}, \\
  Y'' & = \{ V_1\subset V_2,V_3\subset W \mid \dim V_1=r-1, \; \dim V_2=\dim V_3=r \},
\end{align*}
and the morphisms are the obvious ones as before. The bottom right
corners of the diagrams are cartesian squares. As in the previous
proposition, we have
\begin{align*}
  & \E_{r+1}\F_r = (p'_3)_! (p'_{12})_! (p'_{12})^* (p_3)^* [d-1], \\
  & \F_{r-1}\E_r = (p'_3)_! (p''_{23})_! (p''_{23})^* (p_3)^* [d-1].
\end{align*}

Let $i: \Delta\to Y$ be the inclusion of the diagonal. We have
functor isomorphisms
\begin{equation}\label{eqn:effe1}
\begin{split}
  & \Id = (p'_3)_! i_! i^* (p_3)^* = p'_{3!} (i_!\C_\Delta \otimes p_3^* -), \\
  & \E_{r+1}\F_r = p'_{3!} (p'_{12!}\C_{Y'}[d-1] \otimes p_3^* -), \\
  & \F_{r-1}\E_r = p'_{3!} (p''_{23!}\C_{Y''}[d-1] \otimes p_3^* -).
\end{split}
\end{equation}

Note that the varieties $Y',Y''$ are smooth, but $Y$ may be singular
at the diagonal $\Delta$. Indeed, the morphisms $p'_{12}$,
$p''_{23}$ are isomorphisms away from $\Delta$ and are $\BP^{d-r-1}$
respectively $\BP^{r-1}$ fibrations over $\Delta$. By proper base
change we have
\begin{equation*}
  p'_{12!}\C_{Y'}|_{Y\setminus\Delta} \cong p''_{23!}\C_{Y''}|_{Y\setminus\Delta} \cong \C_{Y\setminus\Delta},
\end{equation*}
and
\begin{equation*}
  p'_{12!}\C_{Y'}|_\Delta \cong \bigoplus\limits_{j=0}^{d-r-1} \C_\Delta[-2j], \quad \quad
  p''_{23!}\C_{Y''}|_\Delta \cong \bigoplus\limits_{j=0}^{r-1} \C_\Delta[-2j].
\end{equation*}
On the other hand, $p'_{12!}\C_{Y'}$ and $p''_{23!}\C_{Y''}$ are
both semisimple complexes by the decomposition theorem. Therefore,
by the classification of simple perverse sheaves, both has
$\IC(Y)[-\dim Y]$ as a direct summand and
\begin{equation}\label{eqn:effe2}
  p'_{12!}\C_{Y'} \oplus \bigoplus\limits_{d-r\le j<r} i_!\C_\Delta[-2j]
  \cong p''_{23!}\C_{Y''} \oplus \bigoplus\limits_{r\le j<d-r} i_!\C_\Delta[-2j].
\end{equation}
Combining isomorphisms \eqref{eqn:effe1} and \eqref{eqn:effe2}, we
prove the proposition.
\end{proof}

\begin{exam}
Revisit the case $\bd=(2,2)$ and $r=2$. We may derive the
isomorphism $\E_{r+1}\F_r \IC(\overline{X}_{(2,0)}) \cong
\IC(\overline{X}_{(1,1)})$ as follows. $X_{(2,0)}$ is a point and
$p_3^{-1}(X_{(2,0)})$ is isomorphic to the singular Schubert variety
$\overline{X}_{(1,1)}$. Therefore,
$(p_3)^*\IC(\overline{X}_{(2,0)})[3]$ is a shifted constant sheaf
supported on $p_3^{-1}(X_{(2,0)})$ thus of course not semisimple.
But $p'_{12}$ resolves the singularity of $p_3^{-1}(X_{(2,0)})$.
Thus $(p'_{12})_! (p'_{12})^* (p_3)^*\IC(\overline{X}_{(2,0)})[3]$
is the simple perverse sheaf supported on $p_3^{-1}(X_{(2,0)})$,
sent by $(p'_3)_!$ to $\IC(\overline{X}_{(1,1)})$. Indeed, in this
case $p'_{12!}\C_{Y'} \cong p''_{23!}\C_{Y''} \cong \IC(Y)[-\dim
Y]$.
\end{exam}

Comparing the above propositions with the defining relations
\eqref{eqn:r1}, \eqref{eqn:r2} of $U_\A$, we obtain the
categorification theorem.

\begin{thm}
The endofunctors $\K, \E^{(n)}, \F^{(n)}$ categorify the generators
$K$, $E^{(n)}$, $F^{(n)}$ of $U_\A$, thus endow the Grothendieck
group $Q_\bd$ with a $U_\A$-module structure. More precisely, the
followings hold for $C \in \CQ_\bd$.
\begin{align*}
  & [\K\E C]=q^2[\E\K C], \quad [\K\F C]=q^{-2}[\F\K C], \\
  & [\E\F C]-[\F\E C]=\frac{[\K C]-[\K^{-1}C]}{q-q^{-1}}, \\
  & [\E^{(n)}C]=\frac{[\E^nC]}{[n]_q!}, \quad [\F^{(n)}C]=\frac{[\F^nC]}{[n]_q!}.
\end{align*}
\end{thm}

\begin{rem}
Our categorification described above is a rather rough one, but has
the advantage of being clear and simple. Indeed this has been enough
if we are only concerned with the representations of $U_q(sl_2)$. To
give a more rigorous treatment, one may consider instead the algebra
$\dot{U}_\A$ (cf. \cite{Lu93}) which is a free $\A$-module generated
by the symbols
$$F^{(n)}E^{(m)}1_r, \quad m,n\ge0, \; r\in\Z$$
and subjects to the multiplication
\begin{equation*}
\begin{split}
  & F^{(n)}E^{(m)}1_r \cdot F^{(k)}E^{(l)}1_s
  \\ & \quad
  = \delta_{r,s+2l-2k} \sum_{0\le t\le m,k}
  \bbinom{2l+s}{t}_q \bbinom{n+k-t}{n}_q \bbinom{m+l-t}{l}_q
  F^{(n+k-t)}E^{(m+l-t)}1_s.
\end{split}
\end{equation*}
What follows then is straightforward: associate to
$F^{(n)}E^{(m)}1_r$ the functor $\F^{(n)}\E^{(m)}_\frac{d-r}{2}$ and
use the decomposition theorem to establish functor isomorphisms
mimicing the above multiplication, for example, as we have done for
the following fundamental cases
\begin{equation*}
\begin{split}
  & F^{(0)}E^{(n-1)}1_{r+2} \cdot F^{(0)}E^{(1)}1_r
  = [n]_q F^{(0)}E^{(n)}1_r, \\
  & F^{(n-1)}E^{(0)}1_{r-2} \cdot F^{(1)}E^{(0)}1_r
  = [n]_q F^{(n)}E^{(0)}1_r, \\
  & F^{(0)}E^{(1)}1_{r-2} \cdot F^{(1)}E^{(0)}1_r
  = F^{(1)}E^{(1)}1_r + [r]_q F^{(0)}E^{(0)}1_r.
\end{split}
\end{equation*}
\end{rem}

\subsection{Inner product}\label{sec:mod:ip}

In this subsection we endow the $U_\A$-module $Q_\bd$ with an inner
product by using the bifunctor $\Ext^\bullet_\D(-,D-)$ where
\begin{equation}
  \D = \D(\sqcup_r X_d^r) = \oplus_r \D(X_d^r).
\end{equation}

By \ref{sec:pre:perverse}.(1)(2), there exists a unique symmetric
bilinear form
\begin{equation}
  (,): Q_\bd\times Q_\bd\to\A
\end{equation}
such that
\begin{equation}
  \big([C],[C']\big) = \sum_k \dim\Ext^k_\D(C,DC') \cdot q^{-k}
\end{equation}
for $C,C'\in\CQ_\bd$.

\begin{rem}
Note the identities
$$\big([C],[C']\big) = \sum_k \dim\BH^kD(C\otimes C') \cdot q^{-k} = \sum_k \dim\BH_c^k(C\otimes C') \cdot q^k.$$
\end{rem}

\begin{prop}
For $b_\br=[\IC(\overline{X}_\br)]$,
$b_\bs=[\IC(\overline{X}_{\bs})]$,
$X_\br,X_{\bs}\in\sqcup_r\CS_\bd^r$ we have
\begin{equation}
  (b_\br,b_\bs) \in \delta_{\br,\bs}+q^{-1}\Z_{\ge0}[q^{-1}].
\end{equation}
In particular, the bilinear form $(,)$ on $Q_\bd$ is non-degenerate.
\end{prop}

\begin{proof}
Since $\IC(\overline{X}_\br), \IC(\overline{X}_{\bs})$ are self dual
simple perverse sheaves, by \ref{sec:pre:perverse}.(3)(4)
$\Ext^k_\D(\IC(\overline{X}_\br),D\IC(\overline{X}_{\bs}))$ vanishes
for $k<0$ and has dimension $\delta_{\br,\bs}$ for $k=0$. This
proves the main claim of the proposition. The non-degeneracy follows
from the observation that the bilinear form on the canonical basis
\eqref{eqn:cb} produces a unit matrix modulo $q^{-1}$.
\end{proof}

\begin{cor}\label{cor:iso}
The following conditions are equivalent for $C,C'\in\CQ_\bd$.
\begin{enumerate}
\renewcommand{\theenumi}{\roman{enumi}}
\setlength{\itemsep}{.3ex}
  \item $C\cong C'$.
  \item $\big([C],[C'']\big) = \big([C'],[C'']\big)$ for all $C''\in\CQ_\bd$.
\end{enumerate}
\end{cor}

\begin{prop}\label{prop:ext}
There are bifunctor isomorphisms
\begin{align*}
  & \Ext^\bullet_\D(\K-,D-) = \Ext^\bullet_\D(-,D\K-), \\
  & \Ext^\bullet_\D(\E^{(n)}-,D-) = \Ext^\bullet_\D(-,D\K^n\F^{(n)}[-n^2]-), \\
  & \Ext^\bullet_\D(\F^{(n)}-,D-) = \Ext^\bullet_\D(-,D\K^{-n}\E^{(n)}[-n^2]-).
\end{align*}
\end{prop}

\begin{proof}
The first isomorphism is obvious and the proofs of the next two are
similar. The third one follows from the natural isomorphisms for
$C\in\D(X_d^r), C'\in\D(X_d^{r+n})$
\begin{align*}
  & \Ext^\bullet_{\D(X_d^{r+n})} (\F^{(n)}_rC,DC') \\
  = \;& \Ext^\bullet_{\D(X_d^{r+n})} (p'_!p^*[n(d-n-r)]C,DC') \\
  = \;& \Ext^\bullet_{\D(X_d^r)} (C,p_*p'^![-n(d-n-r)]DC') \\
  = \;& \Ext^\bullet_{\D(X_d^r)} (C,Dp_!p'^*[n(d-n-r)]C') \\
  = \;& \Ext^\bullet_{\D(X_d^r)} (C,D\K^{-n}_r\E^{(n)}_{r+n}[-n^2]C')
  \tag*\qed
\end{align*}
\renewcommand{\qed}{}
\end{proof}

Summarizing, we obtain

\begin{thm}
The bilinear form $(,)$ is an inner product of the $U_\A$-module
$Q_\bd$.
\end{thm}

\begin{exam}\label{exam:d2}
For $\bd=(d)$, we have
\begin{align*}
  & \big([\IC(X_d^r)],[\IC(X_d^r)]\big)
  = \sum_k \dim\Ext^k_{\D(X_d^r)} (\IC(X_d^r),D\IC(X_d^r)) \cdot q^{-k} \\
  & \quad \quad \quad \quad \quad = \sum_k \dim H^k(X_d^r,\C) \cdot q^{-k}
  = \bbinom{d}{r}_qq^{-r(d-r)},
\end{align*}
which agrees with \eqref{eqn:pld}.
\end{exam}

\subsection{Verdier duality and bar involution}\label{sec:mod:d}

Let $\bar{\;}: U_\A \to U_\A$ denote the $\Z$-algebra isomorphism
determined by
\begin{equation}
  \bar{q}=q^{-1}, \quad
  \bar{K}=K^{-1}, \quad
  \bar{E}=E, \quad
  \bar{F}=F.
\end{equation}
The following proposition shows that the Verdier duality functor
categorifies the bar involution of $U_\A$.

\begin{prop}\label{prop:verdier}
We have functor isomorphisms
\begin{equation*}
  D[-1]=[1]D, \quad D\K=\K^{-1}D, \quad D\E^{(n)}=\E^{(n)}D, \quad D\F^{(n)}=\F^{(n)}D.
\end{equation*}
\end{prop}

\begin{proof}
We only prove the last isomorphism. Keep the notation
\eqref{eqn:ef}. Since $p: X_d^{r,r+n} \to X_d^r$ is an
$X_{d-r}^n$-bundle and $p': X_d^{r,r+n} \to X_d^{r+n}$ is proper, we
have
$$Dp^*[n(d-n-r)] = p^*[n(d-n-r)]D, \quad Dp'_! = p'_!D.$$
The isomorphism $D\F^{(n)}_r = \F^{(n)}_rD$ then follows.
\end{proof}

This gives us immediately

\begin{thm}\label{thm:psi}
The Verdier duality functor $D$ induces an anti-$\A$-linear
isomorphism $\Psi: Q_\bd \to Q_\bd$, which satisfies
\begin{enumerate}
\setlength{\itemsep}{.3ex}
  \item $\Psi(b_\br) = b_\br$, for $b_\br=[\IC(\overline{X}_\br)]$,
  $X_\br\in\sqcup_r\CS_\bd^r$;
  \item $\Psi^2=\Id$;
  \item $\Psi(xu) = \bar{x}\Psi(u)$, for $x\in U_\A$, $u \in Q_\bd$.
\end{enumerate}
\end{thm}

Combining Proposition \ref{prop:ext} and Proposition
\ref{prop:verdier} we also obtain

\begin{prop}\label{prop:adjoint}
The functors $\K, \E^{(n)}, \F^{(n)}$ have the functors
$$\K^{-1}, \quad \K^n\F^{(n)}[-n^2], \quad \K^{-n}\E^{(n)}[-n^2]$$
as left adjoints and have the functors
$$\K^{-1}, \quad \K^{-n}\F^{(n)}[n^2], \quad \K^n\E^{(n)}[n^2]$$
as right adjoints, respectively.
\end{prop}

\subsection{Abelian categorification}\label{sec:mod:abel}

In this subsection, we categorify the $U$-modules $\Q(q)\otimes_\A
Q_\bd$ via abelian categories.

First, we realize the additive category $\CQ_\bd$ as a full
subcategory of an abelian category. We have a finite-dimensional
graded $\C$-algebra
\begin{equation}
  A^\bullet = \Ext^\bullet_\D(L,L) = \oplus_r \Ext^\bullet_\D(L^r,L^r)
\end{equation}
where
\begin{equation}
  \D = \D(\sqcup_r X_d^r) = \oplus_r \D(X_d^r)
\end{equation}
and
\begin{equation}
  L = \oplus_r \oplus_{X_\br\in\CS_\bd^r} \IC(\overline{X}_\br),
  \quad L^r = \oplus_{X_\br\in\CS_\bd^r} \IC(\overline{X}_\br).
\end{equation}
The complex $L$ is by definition the direct sum of the simple
perverse sheaves (up to isomorphism) from $\CQ_\bd$. The
multiplication of $A^\bullet$ is given by
\begin{equation*}
\begin{split}
  & \Ext^n_\D(L,L) \otimes \Ext^m_\D(L,L)
  = \Hom_\D(L,L[n]) \otimes \Hom_\D(L[n],L[n+m]) \\
  & \to \Hom_\D(L,L[n+m]) = \Ext^{n+m}_\D(L,L).
\end{split}
\end{equation*}
Then for every complex $C\in\CQ_\bd$,
$$\Ext^\bullet_\D(L,C) \quad \text{(resp. $\Ext^\bullet_\D(C,L)$)}$$
defines a graded left (resp. right) $A^\bullet$-module.

Let $A^\bullet\mof$ denote the category of finite-dimensional graded
left $A^\bullet$-modules and let $A^\bullet\pmof$ denote the full
subcategory consisting of the projectives. By
\ref{sec:pre:perverse}.(3)(4) we have
\begin{enumerate}
\setlength{\itemsep}{.3ex}
\item
$A^\bullet$ is $\Z_{\ge0}$-graded.

\item
$A^0 = \oplus_r \oplus_{X_\br\in\CS_\bd^r}
\Hom_\D(\IC(\overline{X}_\br),\IC(\overline{X}_\br))$, with each
summand isomorphic to $\C$.

\setcounter{enumx}{\value{enumi}}
\end{enumerate}
These further imply
\begin{enumerate}
\setlength{\itemsep}{.3ex} \setcounter{enumi}{\value{enumx}}
\item
The units of the $\C$-summands of $A^0$ are the indecomposable
idempotents of $A^\bullet$.

\item
The $\C$-summands of $A^0$ enumerate the simple left
$A^\bullet$-modules.

\item
$\Ext^\bullet_\D(L,\IC(\overline{X}_\br))$,
$X_\br\in\sqcup_r\CS_\bd^r$ enumerate the indecomposable projective
left $A^\bullet$-modules.

\item
$\Hom_{\A^\bullet\mof}(\Ext^\bullet_\D(L,C), \Ext^\bullet_\D(L,C'))
= \Hom_\D(C,C')$ for $C,C'\in\CQ_\bd$.

\setcounter{enumx}{\value{enumi}}
\end{enumerate}
Therefore, we obtain

\begin{prop}\label{prop:Amof}
The obvious functor $\CQ_{\bd} \to A^\bullet\pmof$ is an equivalence
of categories. Moreover, the equivalence identifies the Grothendieck
group of $A^\bullet\mof$ with $\Q(q)\otimes_\A Q_\bd$.
\end{prop}

In the proposition, the Grothendieck group of the abelian category
$A^\bullet\mof$ means the $\Q(q)$-linear space defined by the
generators each for an isomorphism class of objects from
$A^\bullet\mof$ and the relations
\begin{enumerate}
\renewcommand{\theenumi}{\roman{enumi}}
\setlength{\itemsep}{.3ex}
  \item $[M^\bullet]=[M'^\bullet]+[M''^\bullet]$,
  for exact sequence $M'^\bullet \hookrightarrow M^\bullet \twoheadrightarrow M''^\bullet$;
  \item $[M^{\bullet+1}]=q^{-1}[M^\bullet]$, for $M^\bullet\in A^\bullet\mof$.
\end{enumerate}

Next, we translate the endofunctors $\K, \K^{-1}, \E, \F$ into exact
endofunctors of $A^\bullet\mof$. Recall that every graded
$A^\bullet$-bimodule defines an endofunctor of $A^\bullet\mof$ by
tensoring on the left.

For an additive endofunctor $\G: \CQ_\bd \to \CQ_\bd$, compatible
with the shift functor, there is a well-defined graded
$A^\bullet$-bimodule
\begin{equation}
  \G^\bullet = \Ext^\bullet_\D(L,\G L)
\end{equation}
of which the bimodule structure is given by
$$a\cdot x\cdot b = ax\G(b) \quad \text{for $x\in\G^\bullet$, $a,b\in A^\bullet$}.$$
The followings are easy to verify.
\begin{enumerate}
\setlength{\itemsep}{.3ex} \setcounter{enumi}{\value{enumx}}
\item
$\G^\bullet \otimes_{A^\bullet} \Ext^\bullet_\D(L,C) =
\Ext^\bullet_\D(L,\G C)$ for $C \in \CQ_\bd$. In particular,
$\G^\bullet$ induces an endofunctor of $A^\bullet\pmof$.

\item
If $\G$ has a left adjoint then the $A^\bullet$-bimodule
$\G^\bullet$ is (left and right) projective, hence is flat and
defines an exact endofunctor of $A^\bullet\mof$.
\end{enumerate}

As endofunctors of $\CQ_\bd$, $\K, \K^{-1}, \E, \F$ are compatible
with the shift functor and have left adjoints (Proposition
\ref{prop:adjoint}), therefore they define projective graded
$A^\bullet$-bimodules $\K^\bullet, \K^{-1\bullet}, \E^\bullet,
\F^\bullet$ and hence exact endofunctors of $A^\bullet\mof$. The
results from Section \ref{sec:mod:cat} are then translated to

\begin{thm}
We have isomorphisms of projective graded $A^\bullet$-bimodules
\begin{align*}
  & \K^\bullet \otimes \K^{-1\bullet}
  \cong \K^{-1\bullet} \otimes \K^\bullet \cong A^\bullet, \\
  & \K^\bullet \otimes \E^\bullet
  \cong A^{\bullet-2} \otimes \E^\bullet \otimes \K^\bullet, \\
  & \K^\bullet \otimes \F^\bullet
  \cong A^{\bullet+2} \otimes \F^\bullet \otimes \K^\bullet,
\end{align*}
and
\begin{align*}
  & \E^\bullet \otimes \F^\bullet \oplus \bigoplus\limits_{r>d/2}
  \bigoplus\limits_{j=0}^{(2r-d)-1} \Ext^{\bullet+(2r-d)-1-2j}_\D(L^r,L^r) \\
  & \cong \F^\bullet \otimes \E^\bullet \oplus \bigoplus\limits_{r<d/2}
  \bigoplus\limits_{j=0}^{(d-2r)-1} \Ext^{\bullet+(d-2r)-1-2j}_\D(L^r,L^r).
\end{align*}
Therefore, the abelian category $A^\bullet\mof$ together with the
exact endofunctors $\K^\bullet$, $\K^{-1\bullet}$, $\E^\bullet$,
$\F^\bullet$ categorifies the $U$-module $\Q(q)\otimes_\A Q_\bd$.
\end{thm}

\begin{exam}
For $\bd=(d)$, we have $A^\bullet=\oplus_{0\le r\le d}
H^\bullet(X_d^r,\C)$ and
\begin{align*}
  \K^\bullet & = \oplus_r H^{\bullet-d+2r}(X_d^r,\C), \\
  \E^\bullet & = \oplus_r H^{\bullet+d-r-1}(X_d^{r,r+1},\C), \\
  \F^\bullet & = \oplus_r H^{\bullet+r}(X_d^{r,r+1},\C),
\end{align*}
in which $H^\bullet(X_d^{r,r+1},\C)$ is regarded as a graded
$H^\bullet(X_d^r,\C)$-$H^\bullet(X_d^{r+1},\C)$-bimodule for
$\E^\bullet$ and a graded
$H^\bullet(X_d^{r+1},\C)$-$H^\bullet(X_d^r,\C)$-bimodule for
$\F^\bullet$.
\end{exam}

\subsection{The functor Res}\label{sec:mod:res}

Keep the notations of Section \ref{sec:mod:Q}. We split the
composition $\bd=(d_1,d_2,\dots,d_l)$ into a couple of compositions
\begin{equation}
  \bd'=(d_1,d_2,\dots,d_{l'}), \quad \bd''=(d_{l'+1},d_{l'+2},\dots,d_l)
\end{equation}
of $d'=\sum_{i=1}^{l'}d_i$, $d''=\sum_{i=l'+1}^{l}d_i$ respectively.

Set $W'=W_{d'}$ and $W''=W/W_{d'}$, either inheriting a complete
flag from $W$:
\begin{align}
  & 0=W_0\subset W_1\subset\cdots\subset W_{d'}=W', \\
  & 0=W_{d'}/W_{d'}\subset W_{d'+1}/W_{d'}\subset\cdots\subset W_d/W_{d'}=W''.
\end{align}
We associate a collection of data $X_{d'}^r, P_{\bd'}, \CS_{\bd'}^r$
resp. $X_{d''}^r, P_{\bd''}, \CS_{\bd''}^r$ to $W'$ resp. $W''$ as
in Section \ref{sec:mod:Q}. In this way, $P_{\bd'}\times P_{\bd''}$
is regarded as a quotient of $P_\bd$ and we have
$$\M_{P_\bd} (X_{d'}^{r'}\times X_{d''}^{r''})
  = \M_{P_{\bd'}\times P_{\bd''}} (X_{d'}^{r'}\times X_{d''}^{r''}).
$$

Let $\CQ_{\bd',\bd''}^{r',r''}$ be the full subcategory of
$\D(X_{d'}^{r'}\times X_{d''}^{r''})$ whose objects are the
$P_{\bd'}\times P_{\bd''}$-equivariant semisimple complexes. Up to
isomorphism the objects from $\CQ_{\bd',\bd''}^{r',r''}$ are finite
direct sums of $\IC(\overline{X_{\br'} \times X_{\br''}})[j]$ for
various $P_{\bd'}\times P_{\bd''}$-orbits $X_{\br'}\times X_{\br''}$
and integers $j$. Then set
\begin{equation}
  \CQ_{\bd',\bd''} = \bigoplus_{r',r''} \CQ_{\bd',\bd''}^{r',r''}.
\end{equation}

In the same way as we have done for $\CQ_\bd$, we can define
endofunctors
$$\K', \E'^{(n)}, \F'^{(n)} \quad\text{and}\quad \K'', \E''^{(n)}, \F''^{(n)}$$
of $\CQ_{\bd',\bd''}$ which categorify the generators
$$K\otimes1, E^{(n)}\otimes1, F^{(n)}\otimes1 \quad\text{and}\quad 1\otimes K, 1\otimes E^{(n)}, 1\otimes F^{(n)}$$
of $U_\A \otimes U_\A$, so as to endow the Grothendieck group
$Q_{\bd',\bd''}$ of $\CQ_{\bd',\bd''}$ with a $U_\A$-module
structure. We can also define an inner product in terms of Ext
groups. The $U_\A$-module $Q_{\bd',\bd''}$ has a basis
\begin{equation}
\begin{split}
  & b_{\br',\br''}
  = [\IC(\overline{X_{\br'}\times X_{\br''}})]
  = [\IC(\overline{X}_{\br'}) \boxtimes \IC(\overline{X}_{\br''})], \\
  & \quad \quad \quad \quad \quad
  X_{\br'}\in\sqcup_r\CS_{\bd'}^{r}, \; X_{\br''}\in\sqcup_r\CS_{\bd''}^r.
\end{split}
\end{equation}

Notice the equivalence of categories
\begin{equation}
  \CQ_{\bd'} \times \CQ_{\bd''} \to \CQ_{\bd',\bd''}, \quad (C',C'') \mapsto C'\boxtimes C''.
\end{equation}
It follows that $Q_{\bd'} \otimes Q_{\bd''}$ together with all
algebraic structures defined on it is naturally identified with
$Q_{\bd',\bd''}$. In particular, the basis element $b_{\br'}\otimes
b_{\br''}$ is identified with $b_{\br',\br''}$.

\smallskip

Now we define the functor Res by using the diagram
\begin{equation}\label{eqn:res}
  \xymatrix@1{
  X_d^{r'+r''} & \ar[l]_{\;\;\;\iota} Y^{r',r''} \ar[r]^-\pi & X_{d'}^{r'}\times X_{d''}^{r''}
  }
\end{equation}
where
$$Y^{r',r''} = \{ V\in X_d^{r'+r''} \mid \dim(V\cap W')=r' \},$$
and $\iota$ is the inclusion, $\pi(V)=(V\cap W',V/(V\cap W'))$.
Define for each pair of integers $r',r''$ a functor
\begin{equation}
  \Res_{d',d''}^{r',r''} = \pi_!\iota^*[(d'-r')r'']: \D(X_d^{r'+r''}) \to \D(X_{d'}^{r'}\times X_{d''}^{r''})
\end{equation}
and assemble them together
\begin{equation}
  \Res_{d',d''} = \oplus_{r',r''}\Res_{d',d''}^{r',r''} :
  \oplus_r \D(X_d^r) \to \oplus_{r',r''} \D(X_{d'}^{r'}\times X_{d''}^{r''}).
\end{equation}

The following proposition states that $\Res_{d',d''}$ induces a
functor $\CQ_\bd \to \CQ_{\bd',\bd''}$, hence induces an $\A$-linear
map
\begin{equation}
  \Upsilon_{d',d''}: Q_\bd \to Q_{\bd'}\otimes Q_{\bd''}.
\end{equation}

\begin{prop}
We have $\Res_{d',d''}^{r',r''} C \in \CQ_{\bd',\bd''}^{r',r''}$ for
$C \in \CQ_\bd^{r'+r''}$.
\end{prop}

The proof is immediate from the next two lemmas.

\begin{lem}\label{lem:pi1}
The functor $\pi_![(d'-r')r'']$ and the functor $\pi^*[(d'-r')r'']$
define an equivalence of the categories $\M_{P_\bd} (Y^{r',r''})$,
$\M_{P_{\bd'}\times P_{\bd''}} (X_{d'}^{r'}\times X_{d''}^{r''})$.
\end{lem}

\begin{proof}
Since $\pi$ is a $\C^{(d'-r')r''}$-bundle, $\pi^*[(d'-r')r'']$
induces a fully faithful functor from $\M(X_{d'}^{r'}\times
X_{d''}^{r''})$ to $\M(Y^{r',r''})$. Moreover, the kernel of the
group homomorphism $P_\bd \to P_{\bd'}\times P_{\bd''}$ acts
transitively on each fiber of $\pi$. Hence $\pi^*[(d'-r')r'']$
defines an equivalence of the categories $\M_{P_{\bd'}\times
P_{\bd''}} (X_{d'}^{r'}\times X_{d''}^{r''})$, $\M_{P_\bd}
(Y^{r',r''})$.

On the other hand, $\pi_!\pi^* = [-2(d'-r')r'']$. Thus the functor
$\pi_![(d'-r')r'']$ gives an inverse for the equivalence.
\end{proof}

\begin{lem}\label{lem:pi2}
$\iota^*C$ is a $P_\bd$-equivariant semisimple complex for $C \in
\CQ_\bd^{r'+r''}$.
\end{lem}

\begin{proof}
We may assume $C=\IC(\overline{X}_w)$ where $X_w$ is a Schubert cell
of $X_d^{r'+r''}$. By Lemma \ref{lem:desingular} there exists a
resolution of singularities $f: Z \to \overline{X}_w$ such that
$$Z'=f^{-1}(Y^{r',r''})$$
is a smooth variety. Then $C$ is a direct summand of $f_!\C_Z[\dim
Z]$ and $\iota^*C$ is therefore a direct summand of
$\iota^*f_!\C_Z[\dim Z]$. By the decomposition theorem and proper
base change
$$\iota^*f_!\C_Y[\dim Z] = (f|_{Z'})_!\C_{Z'}[\dim Z]$$
is a semisimple complex. It follows that $\iota^*C$ is a semisimple
complex, whose $P_\bd$-equivariance is obvious.
\end{proof}

\begin{exam}
For $\bd=(1,1,1)$, the $\A$-linear map $\Upsilon_{1,2}$ at level
$r'+r''=1$ is as follows.
\begin{align*}
  & b_{(1,0,0)} \mapsto b_{(1)} \otimes b_{(0,0)}, \\
  & b_{(0,1,0)} \mapsto b_{(0)} \otimes b_{(1,0)} + q^{-1} b_{(1)} \otimes b_{(0,0)}, \\
  & b_{(0,0,1)} \mapsto b_{(0)} \otimes b_{(0,1)} + q^{-2} b_{(1)} \otimes b_{(0,0)}.
\end{align*}
\end{exam}

\subsection{The isomorphism $Q_\bd \cong Q_{\bd'}\otimes Q_{\bd''}$}\label{sec:mod:split}

Recall that we split the composition $\bd=(d_1,d_2,\dots,d_l)$ into
$$\bd'=(d_1,d_2,\dots,d_{l'}), \quad \bd''=(d_{l'+1},d_{l'+2},\dots,d_l).$$
In the next proposition, we associate to each composition
$\br=(r_1,r_2,\dots,r_l)$ a pair of compositions
$$\br'=(r_1,r_2,\dots,r_{l'}), \quad \br''=(r_{l'+1},r_{l'+2},\dots,r_l).$$

\begin{prop}\label{prop:basis}
For $b_\br=[\IC(\overline{X}_\br)]$, $X_\br\in\CS_\bd^r$ we have
\begin{equation}
  \Upsilon_{d',d''}(b_\br)
  = b_{\br'} \otimes b_{\br''} +
  \sum_{X_\bs\in\CS_\bd^r: \; \overline{X}_\bs \subsetneqq \overline{X}_\br}
  c_{\br,\bs} \cdot b_{\bs'} \otimes b_{\bs''}
\end{equation}
where $c_{\br,\bs} \in q^{-1}\Z_{\ge0}[q^{-1}]$. Therefore,
$\Upsilon_{d',d''}: Q_\bd \to Q_{\bd'}\otimes Q_{\bd''}$ is an
$\A$-linear isomorphism.
\end{prop}

\begin{proof}
We set
$$c_{\br,\bs} = \sum_k n_{\br,\bs}^k \cdot q^k$$
where $n_{\br,\bs}^k$ are the multiplicities appearing in the
decomposition
$$\Res_{d',d''} \IC(\overline{X}_\br)
  \cong \bigoplus_k \bigoplus_{X_\bs\in\CS_\bd^r}
  n_{\br,\bs}^k \cdot \IC(\overline{X}_{\bs'})\boxtimes\IC(\overline{X}_{\bs''})[-k].
$$
Then
$$\Upsilon_{d',d''}(b_\br) = \sum_{X_\bs\in\CS_\bd^r} c_{\br,\bs} \cdot b_{\bs'} \otimes b_{\bs''}.$$
The simple perverse sheaf $\IC(\overline{X}_\br)$ is by definition
the intermediate extension of the shifted constant sheaf
$\C_{X_\br}[\dim X_\br]$. It follows that
\begin{enumerate}
\renewcommand{\theenumi}{\roman{enumi}}
\setlength{\itemsep}{.3ex}
\item
$\IC(\overline{X}_\br)|_{X_\br} = \C_{X_\br}[\dim X_\br]$;
\item
$\IC(\overline{X}_\br)|_{X_\bs} = 0$ if $\overline{X}_\bs
\not\subset \overline{X}_\br$; and
\item
$\pH^k(\IC(\overline{X}_\br)|_{X_\bs}) = 0$ for $k\ge0$ if
$\overline{X}_\bs \subsetneqq \overline{X}_\br$.
\end{enumerate}
Further, by proper base change,
\begin{enumerate}
\renewcommand{\theenumi}{\roman{enumi}}
\setlength{\itemsep}{.3ex}
\item
$\Res_{d',d''} \IC(\overline{X}_\br)|_{X_{\br'}\times X_{\br''}} =
\C_{X_{\br'}\times X_{\br''}}[\dim X_{\br'}\times X_{\br''}]$;
\item
$\Res_{d',d''} \IC(\overline{X}_\br)|_{X_{\bs'}\times X_{\bs''}}=0$
if $\overline{X}_\bs \not\subset \overline{X}_\br$; and
\item
$\pH^k(\Res_{d',d''} \IC(\overline{X}_\br)|_{X_{\bs'}\times
X_{\bs''}})=0$ for $k\ge0$ if $\overline{X}_\bs \subsetneqq
\overline{X}_\br$.
\end{enumerate}
This gives $c_{\br,\br}=1$; $c_{\br,\bs}=0$ if $\overline{X}_\bs
\not\subset \overline{X}_\br$; and $c_{\br,\bs} \in
q^{-1}\Z_{\ge0}[q^{-1}]$ if $\overline{X}_\bs \subsetneqq
\overline{X}_\br$.

The claim of isomorphism follows from that the $\A$-linear map
$\Upsilon_{d',d''}$ can be represented by a triangular $\A$-matrix
with unit diagonal.
\end{proof}

\begin{prop}
The $\A$-linear map $\Upsilon_{d',d''}: Q_\bd \to Q_{\bd'}\otimes
Q_{\bd''}$ preserves inner product.
\end{prop}

\begin{proof}
For $C,C'\in\CQ_\bd^r$, keeping the notation \eqref{eqn:res} we have
\begin{align*}
  & \Ext^\bullet_{\D(X_d^r)}(C,DC') \\
  \cong \;& \oplus_{r'+r''=r} \Ext^\bullet_{\D(Y^{r',r''})}(\iota^*C,D\iota^*C') \\
  \cong \;& \oplus_{r'+r''=r} \Ext^\bullet_{\D(Y^{r',r''})}(\pi^*\pi_![2(d'-r')r'']\iota^*C,D\iota^*C') \\
  = \;& \oplus_{r'+r''=r} \Ext^\bullet_{\D(X_{d'}^{r'}\times X_{d''}^{r''})}
  (\pi_!\iota^*[(d'-r')r'']C,D\pi_!\iota^*[(d'-r')r'']C') \\
  = \;& \oplus_{r'+r''=r} \Ext^\bullet_{\D(X_{d'}^{r'}\times X_{d''}^{r''})}
  (\Res_{d',d''}^{r',r''}C,D\Res_{d',d''}^{r',r''}C').
\end{align*}
The first isomorphism is by applying Corollary \ref{cor:flag3} to
the decomposition
$$X_d^r=\sqcup_{r'+r''=r}Y^{r',r''};$$
the second is by Lemma \ref{lem:pi1} and Lemma \ref{lem:pi2}. This
gives
$$\big([C],[C']\big)=\big([\Res_{d',d''}C], [\Res_{d',d''}C']\big).$$
Thus the proposition follows.
\end{proof}

It remains to check the compatibility of $\Upsilon_{d',d''}$ with
the comultiplication \eqref{eqn:comul}.

\begin{prop}
For $C\in\CQ_\bd$ we have isomorphisms
\begin{align*}
  & \Res_{d',d''}\K C \cong \K'\K''\Res_{d',d''}C, \\
  & \Res_{d',d''}\E C \cong (\E'\oplus\K'\E'')\Res_{d',d''}C, \\
  & \Res_{d',d''}\F C \cong (\F'\K''^{-1}\oplus\F'')\Res_{d',d''}C.
\end{align*}
Therefore, the $\A$-linear map $\Upsilon_{d',d''}: Q_\bd \to
Q_{\bd'}\otimes Q_{\bd''}$ is a homomorphism of $U_\A$-modules.
\end{prop}

\begin{proof}
We prove the third isomorphism, which by Corollary \ref{cor:iso} is
equivalent to that the equality
\begin{equation}\label{eqn:iso0}
  \big([\Res_{d',d''}\F C],[C']\big)
  = \big([(\F'\K''^{-1}\oplus\F'')\Res_{d',d''}C], [C']\big)
\end{equation}
holds for all $C'\in\CQ_{\bd',\bd''}$

Suppose $C\in\CQ_\bd^r$ and $C'\in\CQ_{\bd',\bd''}^{r',r''}$. We
assume $r'+r''=r+1$ as well; otherwise both sides of the above
equality vanish. Our first commutative diagram is
$$\xymatrix{
  & & \ar[ld]_j Z \ar[rd]^\rho \\
  X^r_d & \ar[l]_{p} X^{r,r+1}_d \ar[r]^{p'} & X^{r+1}_d
  & \ar[l]_{\iota} Y^{r',r''} \ar[r]^--{\pi} & X_{d'}^{r'}\times X_{d''}^{r''}
}
$$
where
$$Z = \{ (V_1,V_2) \in X_d^{r,r+1} \mid \dim(V_2\cap W')=r' \},$$
$j$ is the inclusion, $\rho(V_1,V_2)=V_2$, and $p,p'$ are given as
\eqref{eqn:ef}, $\iota,\pi$ are given as \eqref{eqn:res}. The middle
part of the diagram is a cartesian square, thus by proper base
change we have
$$\Res_{d',d''}\F C
  = \pi_! \iota^* [(d'-r')r''] p'_! p^* [d-r-1] C
  = \pi_! \rho_! j^* p^* C[k]
$$
where we set $k=(d'-r')r''+(d-r-1)$.

Note that $\rho: Z\to Y^{r',r''}$ is a $\BP^r$-bundle,
$\rho^*\pi^*C'$ is therefore a semisimple complex. Since $Z$ is a
locally closed smooth subvariety of $X_d^{r,r+1}$, $\rho^*\pi^*C'$
is the restriction of a $P_\bd$-equivariant semisimple complex on
$X_d^{r,r+1}$. Applying Corollary \ref{cor:flag3} to the
decomposition $Z=Z_1\sqcup Z_2$ where
\begin{align*}
  & Z_1 = \{ (V_1,V_2) \in Z \mid \dim(V_1\cap W')=r'-1 \}, \\
  & Z_2 = \{ (V_1,V_2) \in Z \mid \dim(V_1\cap W')=r', \},
\end{align*}
we deduce that
\begin{equation}\label{eqn:iso1}
\begin{split}
  & \Ext^\bullet_{\D(X_{d'}^{r'}\times X_{d''}^{r''})} (\Res_{d',d''}\F C, DC') \\
  = & \Ext^\bullet_{\D(Z)} (j^*p^*C[k], D\rho^*\pi^*C') \\
  \cong & \Ext^\bullet_{\D(Z_1)} (j_1^*p^*C[k], D\rho_1^*\pi^*C')
  \oplus \Ext^\bullet_{\D(Z_2)} (j_2^*p^*C[k], D\rho_2^*\pi^*C') \\
  = & \Ext^\bullet_{\D(X_{d'}^{r'}\times X_{d''}^{r''})}
  (\pi_!\rho_{1!}j_1^*p^*C[k] \oplus \pi_!\rho_{2!}j_2^*p^*C[k], DC').
\end{split}
\end{equation}
In the above equation, we set $j_i=j|_{Z_i}$ and
$\rho_i=\rho|_{Z_i}$, $i=1,2$.

Then we consider the following commutative diagrams of which the top
left corners are cartesian squares.
$$\xymatrix{
  X_{d'}^{r'-1}\times X_{d''}^{r''}
  & \ar[l]_{p_1} X_{d'}^{r'-1,r'}\times X_{d''}^{r''} \ar[r]^----{p'_1}
  & X_{d'}^{r'}\times X_{d''}^{r''} \\
  Y^{r'-1,r''} \ar[u]^{\pi_1} \ar[d]_{\iota_1} & \ar[l]^{} Z'_1 \ar[u]^{} & Y^{r',r''} \ar[u]_{\pi} \\
  X^r_d & \ar[l]_{p} X^{r,r+1}_d & \ar[l]_{j_1} Z_1 \ar[u]_{\rho_1} \ar[lu]_{\tilde\rho_1}
}
$$
$$\xymatrix{
  X_{d'}^{r'}\times X_{d''}^{r''-1}
  & \ar[l]_{p_2} X_{d'}^{r'}\times X_{d''}^{r''-1,r''} \ar[r]^----{p'_2}
  & X_{d'}^{r'}\times X_{d''}^{r''} \\
  Y^{r',r''-1} \ar[u]^{\pi_2} \ar[d]_{\iota_2} & \ar[l]^{} Z'_2 \ar[u]^{} & Y^{r',r''} \ar[u]_{\pi} \\
  X^r_d & \ar[l]_{p} X^{r,r+1}_d & \ar[l]_{j_2} Z_2 \ar[u]_{\rho_2} \ar[lu]_{\tilde\rho_2}
}
$$
where
\begin{align*}
  & Z'_1 = \{ V_1\in Y^{r'-1,r''}, V_2\in X_{d'}^{r'} \mid V_1\cap W' \subset V_2 \}, \\
  & Z'_2 = \{ V_1\in Y^{r',r''-1}, V_2\in X_{d''}^{r''} \mid V_1/(V_1\cap W') \subset V_2 \},
\end{align*}
$\tilde\rho_1(V_1,V_2)=(V_1,V_2\cap W')$,
$\tilde\rho_2(V_1,V_2)=(V_1,V_2/(V_2\cap W'))$ and $p_i,p'_i$ are
given as \eqref{eqn:ef}, $\iota_i,\pi_i$ are given as
\eqref{eqn:res}, $i=1,2$. We have
\begin{equation}\label{eqn:iso2}
\begin{split}
  & \pi_! (\rho_1)_! (j_1)^* p^* [k]
  = (p'_1)_! (p_1)^* (\pi_1)_! (\iota_1)^* [k]
  = \F'_{r'-1}\K''^{-1}_{r''} \Res_{d',d''}^{r'-1,r''}, \\
  & \pi_! (\rho_2)_! (j_2)^* p^* [k]
  = (p'_2)_! (p_2)^* (\pi_2)_! (\iota_2)^* [k-2(d'-r')]
  = \F''_{r''-1} \Res_{d',d''}^{r',r''-1}.
\end{split}
\end{equation}
Here we used the isomorphisms
$$(\tilde\rho_1)_!(\tilde\rho_1)^* = \Id, \quad (\tilde\rho_2)_!(\tilde\rho_2)^* = [-2(d'-r')]$$
which follow from that $\tilde\rho_1$ is an isomorphism and
$\tilde\rho_2$ is a $\C^{d'-r'}$-bundle.

Assembling isomorphisms \eqref{eqn:iso1} and \eqref{eqn:iso2}, we obtain
\eqref{eqn:iso0}, hence prove our proposition.
\end{proof}

Above propositions are summarized to

\begin{thm}
The $\A$-linear map $\Upsilon_{d',d''}: Q_\bd \to Q_{\bd'}\otimes
Q_{\bd''}$ induced by the functor $\Res_{d',d''}: \CQ_\bd \to
\CQ_{\bd',\bd''}$ is an inner product preserving isomorphism of
$U_\A$-modules.
\end{thm}

\subsection{The isomorphism $Q_\bd \cong \Lambda_\bd$}\label{sec:mod:ql}

We have constructed in a purely geometric way various
finite-dimensional $U_\A$-modules and established isomorphisms among
them. Now we relate them to the $U_\A$-modules $\Lambda_\bd$
introduced in Section \ref{sec:pre:u}.

Recall that the $U_\A$-module $Q_{(d)}$ has a basis
$$b_{(r)}=[\IC(X_d^r)], \quad 0\le r\le d$$
and the $U_\A$-module $\Lambda_d$ has a basis
$$v_r=\bar{F}^{(r)}, \quad 0\le r\le d.$$
By Example \ref{exam:d1} and Example \ref{exam:d2}, the $\A$-linear
map
\begin{equation}
  \varphi_d: Q_{(d)}\to \Lambda_d, \quad b_{(r)} \mapsto v_r
\end{equation}
is an inner product preserving isomorphism of $U_\A$-modules.

For general cases, we apply the Res functor repeatedly to form an
inner product preserving isomorphism of $U_\A$-modules
\begin{equation}
  \Upsilon_{d_1,d_2,\dots,d_l}: Q_\bd \to
  Q_{(d_1)}\otimes Q_{(d_2)}\otimes\cdots\otimes Q_{(d_l)}.
\end{equation}
Followed by $\varphi_{d_1} \otimes \varphi_{d_2}
\otimes\cdots\otimes \varphi_{d_l}$, it gives rise to an inner
product preserving isomorphism of $U_\A$-modules
\begin{equation}
  \varphi_\bd =
  (\varphi_{d_1} \otimes \varphi_{d_2} \otimes\cdots\otimes \varphi_{d_l})
  \circ \Upsilon_{d_1,d_2,\dots,d_l}: Q_\bd\to \Lambda_\bd.
\end{equation}

In the rest of this subsection, we give a straightforward
description of this isomorphism.

For each $P_\bd$-orbit $X_\br \in \CS_\bd^r$ and for each complex
$C\in\CQ_\bd^r$, there are a set of integers $n_\br^k(C)$ appearing
as multiplicities in the isomorphism
\begin{equation*}
  \pH^k(C|_{X_\br}) \cong n_\br^k(C) \cdot \C_{X_\br}[\dim X_\br],
\end{equation*}
they forming a polynomial
\begin{equation*}
  n_\br(C) = \sum_k n_\br^k(C) \cdot q^k \in \Z_{\ge0}[q,q^{-1}].
\end{equation*}
In the above notations, the isomorphism $\varphi_\bd: Q_\bd\to
\Lambda_\bd$ is such that
\begin{equation}
  \varphi_\bd([C]) = \sum_{X_\br\in\CS_\bd^r} n_\br(C) \cdot v_\br
\end{equation}
for $C\in\CQ_\bd^r$, where $v_\br$ are the elements of $\Lambda_\bd$
defined in \eqref{eqn:vr}.

\begin{exam}
For $\bd=(2,2)$, the isomorphism $\varphi_\bd$ at level $r=2$ is as
follows.
\begin{align*}
  & b_{(2,0)} \mapsto v_2 \otimes v_0, \\
  & b_{(1,1)} \mapsto v_1 \otimes v_1 + (q^{-1}+q^{-3}) v_2 \otimes v_0, \\
  & b_{(0,2)} \mapsto v_0 \otimes v_2 + q^{-1} v_1 \otimes v_1 + q^{-4} v_2 \otimes v_0.
\end{align*}
\end{exam}

The following proposition is a specialization of Proposition
\ref{prop:basis}.

\begin{prop}
For $b_\br=[\IC(\overline{X}_\br)]$, $X_\br\in\CS_\bd^r$ we have
\begin{equation}
  \varphi_\bd(b_\br)
  = v_\br + \sum_{X_{\bs}\in\CS_\bd^r: \; \overline{X}_{\bs} \subsetneqq \overline{X}_\br}
  c_{\br,\bs} \cdot v_{\bs}
\end{equation}
where $c_{\br,\bs} \in q^{-1}\Z_{\ge0}[q^{-1}]$.
\end{prop}

\begin{proof}
Set $c_{\br,\bs}=n_{\bs}(\IC(\overline{X}_\br))$, then
$\varphi_\bd(b_\br) = \sum_{X_{\bs}\in\CS_\bd^r} c_{\br,\bs} \cdot
v_{\bs}$. The simple perverse sheaf $\IC(\overline{X}_\br)$ is by
definition the intermediate extension of the shifted constant sheaf
$\C_{X_\br}[\dim X_\br]$. It follows that
\begin{enumerate}
\renewcommand{\theenumi}{\roman{enumi}}
\setlength{\itemsep}{.3ex}
\item
$\IC(\overline{X}_\br)|_{X_{\br}}=\C_{X_\br}[\dim X_\br]$;
\item
$\IC(\overline{X}_\br)|_{X_{\bs}}=0$ if $\overline{X}_{\bs}
\not\subset \overline{X}_\br$; and
\item
$\pH^k(\IC(\overline{X}_\br)|_{X_{\bs}})=0$ for $k\ge0$ if
$\overline{X}_{\bs} \subsetneqq \overline{X}_\br$.
\end{enumerate}
Therefore, $c_{\br,\br}=1$; $c_{\br,\bs}=0$ if $\overline{X}_{\bs}
\not\subset \overline{X}_\br$; and $c_{\br,\bs} \in
q^{-1}\Z_{\ge0}[q^{-1}]$ if $\overline{X}_{\bs} \subsetneqq
\overline{X}_\br$.
\end{proof}

\begin{rem}
It is not difficult by interpreting the coefficients $c_{\br,\bs}$
as parabolic Kazhdan-Lusztig polynomials \cite{KL79}\cite{Deo87} to
identify the canonical basis \eqref{eqn:cb} of $Q_\bd$ with the one
introduced by Lusztig \cite{Lu93}. Cf. \cite{FKK98}. Moreover, since
the anti-$\A$-linear isomorphism $\Psi$ from Theorem \ref{thm:psi}
is uniquely determined by property (1) of the theorem, it is
therefore the same as the one from \cite{Lu93} defined by means of
quasi-universal $R$-matrix.
\end{rem}

\section{Categorification of $R$-matrices}\label{sec:R}

One remarkable achievement (and impetus) on the topic of
categorification is the discovery of Khovanov homology of knots and
links \cite{Kh00}, which has become of particular interest after
Rasmussen's elementary proof \cite{Ras03} of Milnor's conjecture. In
fact, the only solutions to the conjecture known before are using
gauge theory or Floer homology.

Khovanov homology is able to be realized in many different ways and
has been generalized to the categorification of several other
quantum invariants of knots and links. See for instance
\cite{Str05}, \cite{CK07}; \cite{Kh03}, \cite{KR04}. However, the
quantum invariants under consideration are still very limited, and
the machinery used is apparently hard to be applied to general
cases.

To give a uniform treatment for the categorification of general
quantum invariants, one possible approach is to follow
Reshetikhin-Turaev's principle \cite{RT90} for building tangle
invariants from representations of quantum groups. This means to
categorify, besides representations of quantum groups, $R$-matrices
and ``cup/cap'' homomorphisms among them.

In this section, we deal with the issue of $R$-matrices on
$U_q(sl_2)$-modules. This part of work is new in many aspects.

\smallskip

Formally speaking, a system of $R$-matrices on the $U_\A$-modules
$\Lambda_\bd$ is a collection of $U_\A$-module isomorphisms
$R(\bd,\sigma): \Lambda_\bd \to \Lambda_{\sigma(\bd)}$, each for a
composition $\bd=(d_1,d_2,\dots,d_l)$ and a permutation $\sigma\in
S_l$, such that
\begin{equation}\label{eqn:R:braid}
  R(\sigma_2(\bd),\sigma_1) \circ R(\bd,\sigma_2) = R(\bd,\sigma_1\sigma_2),
\end{equation}
whenever $\ell(\sigma_1\sigma_2) = \ell(\sigma_1) + \ell(\sigma_2)$.

The standard algebraic approach to the realization of $R$-matrices
is by Drinfeld's universal $R$-matrix (cf. \cite[XVII.4.2]{Kas95}
and formula \eqref{eqn:R:univ} below), which assigns to each pair of
$U_\A$-modules $\Lambda_{d_1},\Lambda_{d_2}$ an isomorphism
\begin{equation}
  R_{d_1,d_2}: \Lambda_{d_1}\otimes\Lambda_{d_2} \to \Lambda_{d_2}\otimes\Lambda_{d_1},
\end{equation}
then composes them in the obvious way to give the others.

\smallskip

The major difficulty underlying categorification of $R$-matrices is
the failure of their positivity over canonical basis; that is, a
canonical basis element may be sent to a linear combination in which
both positive and negative coefficients occur (cf. Example
\ref{exam:catR1}). This forces us to settle the categorification
problem by using complexes of functors rather than merely functors.

\smallskip

Section \ref{sec:R:mix} constitutes the heart
part of this section, in which we introduce the notion of pure
resolution of mixed complexes and establish a uniqueness theorem for
mixed perverse sheaves. Then, we categorify the braiding relation
\eqref{eqn:R:braid} in Section \ref{sec:R:R} and establish
categorification theorem in the reminder subsections.

\subsection{Some homological algebra}\label{sec:R:homology}

Below are some elementary facts that will be used in this section.

\begin{lem}\label{lem:homology2}
Suppose we are given a morphism of complexes forming by objects and
morphisms from a triangulated category.
$$\xymatrixrowsep{1.5pc}\xymatrix{
  \cdots \ar[r] & C^1 \ar[r] \ar@{=}[d] & C^2 \ar[r]^a \ar[d]_b & C^3 \ar[r] \ar[d]^c & C^4 \ar[r] \ar@{=}[d] & \cdots \\
  \cdots \ar[r] & C^1 \ar[r] & \tilde{C}^2 \ar[r]^d & \tilde{C}^3 \ar[r] & C^4 \ar[r] & \cdots
}
$$
If there is a triangle morphism
$$\xymatrixrowsep{1.5pc}\xymatrix{
  B \ar[r] \ar@{=}[d] & C^2 \ar[r]^b \ar[d]_a & \tilde{C}^2 \ar[r]^0 \ar[d]^d & \\
  B \ar[r] & C^3 \ar[r]^c & \tilde{C}^3 \ar[r]^0 &
}
$$
then the above complex morphism is a homotopy equivalence.
\end{lem}

\begin{proof}
Rewrite the triangle morphism as follows, in which $e,f$ and $b,c$
are the obvious inclusions and projections, respectively.
$$\xymatrixrowsep{1.5pc}\xymatrix{
  B \ar[r]^---{e} \ar@{=}[d] & B\oplus \tilde{C}^2 \ar[r]^---{b} \ar[d]_a & \tilde{C}^2 \ar[r]^0 \ar[d]^d & \\
  B \ar[r]^---{f} & B\oplus \tilde{C}^3 \ar[r]^---{c} & \tilde{C}^3 \ar[r]^0 &
}
$$
Then $a$ must be in the form \begin{scriptsize} $\begin{pmatrix} \Id
& 0 \\ * & d \end{pmatrix}$\end{scriptsize}. Then a direct
computation.
\end{proof}

\begin{lem}\label{lem:homology4}
Any sequence $C^{\le w-2} \to C^{\le w-1} \to \tilde{C}^{\le w-1}
\to C^{\le w}$ in a triangulated category extends to commutative
diagrams
$$\xymatrixrowsep{.5pc}\xymatrix{
  & C^w[-1] \ar[dr]^a & & C^{w-1} \\
  \ar[dr]^{d'} & & C^{\le w-1} \ar[ur]^b \ar[dr] \ar[dd]_e & & \\
  & C^{\le w-2} \ar[ur] \ar[dr] & & C^{\le w} \ar[ur]^{c'} \ar[dr]^{a'} \\
  \ar[ur]^{b'} & & \tilde{C}^{\le w-1} \ar[ur] \ar[dr]^d & & \\
  & \tilde{C}^w[-1] \ar[ur]^c & & \tilde{C}^{w-1}
}
$$
$$\xymatrixrowsep{1.5pc}\xymatrix{
  & B \ar@{=}[r] \ar[d] & B \ar@{=}[r] \ar[d] & B \ar[d] \\
  C^{\le w} \ar[r]^{a'} \ar@{=}[d] & C^w \ar[r]^--a \ar[d]
  & C^{\le w-1}[1] \ar[r]^b \ar[d]_e & C^{w-1}[1] \ar[d] \ar[r]^{b'} & C^{\le w-2}[2] \ar@{=}[d] \\
  C^{\le w} \ar[r]^{c'} & \tilde{C}^w \ar[r]^--c & \tilde{C}^{\le w-1}[1] \ar[r]^d
  & \tilde{C}^{w-1}[1] \ar[r]^{d'} & C^{\le w-2}[2]
}
$$
in which the vertical and the slash lines are exact triangles.
\end{lem}

\begin{proof}
Follows directly from the defining axioms of triangulated category.
\end{proof}

\begin{lem}[Postnikov system]\label{lem:homology5}
Suppose we are given a system of exact triangles from a triangulated
category
\begin{equation}\label{eqn:ha1}
  C^{\le w-1} \to C^{\le w} \to C^w.
\end{equation}
Then the following sequence
\begin{equation*}
  \cdots \to C^{w+1}[-w-1] \to C^w[-w] \to C^{w-1}[-w+1] \to \cdots
\end{equation*}
in which the morphisms are the compositions
\begin{equation*}
  C^w[-w] \to C^{\le w-1}[-w+1] \to C^{w-1}[-w+1],
\end{equation*}
is a complex.
\end{lem}

\begin{proof}
Observe that the morphisms
$$C^{w+1}[-w-1] \to C^w[-w] \to C^{w-1}[-w+1]$$
are realized by the sequence
\begin{equation*}
   C^{w+1}[-w-1] \to C^{\le w}[-w] \to C^{w}[-w] \to C^{\le w-1}[-w+1] \to C^{w-1}[-w+1],
\end{equation*}
which composes to zero because its middle part is an exact triangle.
\end{proof}

\begin{lem}\label{lem:homology6}
Suppose in addition to the assumption of the above lemma, there are triangle
morphisms
$$\xymatrixrowsep{1.5pc}\xymatrix{
  B^{\le w-1} \ar[r] \ar[d] & C^{\le w-1} \ar[r] \ar[d] & C \ar@{=}[d] \\
  B^{\le w} \ar[r] & C^{\le w} \ar[r] & C
}
$$
Then they extend to give a system of exact triangles
\begin{equation}\label{eqn:ha3}
  B^{\le w-1} \to B^{\le w} \to C^w
\end{equation}
which induce the same complex as \eqref{eqn:ha1}.
\end{lem}

\begin{proof}
By the octahedron axiom of triangulated category, the given triangle morphisms
extend to commutative diagrams with exact triangles on their rows and columns
$$\xymatrixrowsep{1.5pc}\xymatrix{
  B^{\le w-1} \ar[r] \ar[d] & C^{\le w-1} \ar[r] \ar[d] & C \ar@{=}[d] \\
  B^{\le w} \ar[r] \ar[d] & C^{\le w} \ar[d] \ar[r] & C \\
  C^w \ar@{=}[r] & C^w
}
$$
This gives the exact triangles \eqref{eqn:ha3} and, further, commutative
diagrams
$$\xymatrixrowsep{1.5pc}\xymatrix{
  C^w \ar[r] \ar@{=}[d] & B^{\le w-1}[1] \ar[r] \ar[d] & C^{w-1}[1] \ar@{=}[d] \\
  C^w \ar[r] & C^{\le w-1}[1] \ar[r] & C^{w-1}[1]
}
$$
saying that \eqref{eqn:ha3} induce the same complex as
\eqref{eqn:ha1}.
\end{proof}

\subsection{Pure resolution of mixed complexes}\label{sec:R:mix}

Let $\BF_q$ be a finite field with $q$ elements and $\BF$ be its
algebraic closure. Let $X_0$ be a scheme of finite type over $\BF_q$
and let $X$ be the scheme $X_0\times_{Spec(\BF_q)}Spec(\BF)$ over
$\BF$. We denote by $\D(X_0) = \D^b_c(X_0,\bar\Q_l)$ (resp. $\D(X) =
\D^b_c(X,\bar\Q_l)$) the triangulated category of $\bar\Q_l$-sheaves
\cite[2.2.18]{BBD82} on $X_0$ (resp. $X$), where $l$ is a prime
number invertible in $\BF_q$. For a complex $C_0\in\D(X_0)$ we
denote by $C\in\D(X)$ its pullback to $X$.

A complex from $\D(X_0)$ is called mixed if all its cohomology
sheaves are mixed $\bar\Q_l$-sheaves. Mixed complexes from $\D(X_0)$
form a full triangulated subcategory $\D_m(X_0)$. It inherits the
perverse t-structure from $\D(X_0)$ thus gives rise to the category
$\M_m(X_0)$ of mixed perverse sheaves.

We denote by $\D_{\le w}(X_0)$ the full subcategory of $\D_m(X_0)$
consisting of those mixed complexes whose $i$-th cohomology sheaf is
of weight $\le w+i$ for all $i$ and denote by $\D_{\ge w}(X_0)$ the
full subcategory consisting of those mixed complexes $C$ such that
$DC\in\D_{\le -w}(X_0)$. The complexes from $\D_{\le w}(X_0) \cap
\D_{\ge w}(X_0)$ are called pure of weight $w$. Be careful of that
the purity of a mixed $\bar\Q_l$-sheaf does not necessarily agree
with the one as a mixed complex.

\smallskip

Listed below are some properties of mixed complexes (cf.
\cite{BBD82}, \cite{KW01}). The key step to them is the
proof of (3) for $f_!$, which is the main result of \cite{De80}. Be aware of
that the decomposition theorem is immediate from (1)-(3).
\begin{enumerate}
\setlength{\itemsep}{.3ex}
\item
Simple mixed perverse sheaves are pure.

\item
If $C_0\in\D_m(X_0)$ is pure, then $C \cong \oplus_n\pH^n(C)[-n]$
and $\pH^n(C)$ is a semisimple perverse sheaf for all $n$.

\item
For a morphism $f: X_0 \to Y_0$, the functors $f_!,f^*$ preserve
$\D_{\le w}$ and the functors $f_*,f^!$ preserve $\D_{\ge w}$.

\item
For a smooth morphism $f: X_0 \to Y_0$ of relative dimension $d$, we
have $f^*[d]=f^![-d](-d)$, where $(-d)$ indicates the Tate twist
(increasing the weight by $2d$).

\item
The outer tensor product functor $\boxtimes$ sends $\D_{\le
w_1}\times\D_{\le w_2}$ to $\D_{\le w_1+w_2}$ and sends $\D_{\ge
w_1}\times\D_{\ge w_2}$ to $\D_{\ge w_1+w_2}$.

\item
There are exact sequences for $C_0,C'_0\in\D_m(X_0)$
$$\Ext^{n-1}_{\D(X)}(C,C')_F \hookrightarrow \Ext^n_{\D(X_0)}(C_0,C'_0) \twoheadrightarrow \Ext^n_{\D(X)}(C,C')^F$$
where $F$ is the geometric Frobenius.

\item
Assume $C_0\in\D_{\le w}(X_0)$, $C'_0\in\D_{> w}(X_0)$. We have
$\Ext^n_{\D(X)}(C,C')_F = \Ext^n_{\D(X)}(C,C')^F = 0$ for $n\ge0$.
Therefore, $\Ext^n_{\D(X_0)}(C_0,C'_0) = 0$ for $n>0$. Further,
$\Ext^\bullet_{\D(X_0)}(C_0,C'_0) = 0$ provided in addition that
$C_0,C'_0$ are mixed perverse sheaves.

\item
The pullback $\Ext^n_{\D(X_0)}(C_0,C'_0) \to \Ext^n_{\D(X)}(C,C')$
is the zero map for $C_0\in\D_{\le w}(X_0)$, $C'_0\in\D_{\ge
w}(X_0)$ and $n>0$.

\item
A subquotient of a mixed perverse sheaf of weight $\le w$ (resp.
$\ge w$) is of weight $\le w$ (resp. $\ge w$).

\item
A mixed perverse sheaf $C_0\in\D_m(X_0)$ admits a unique weight
filtration $W^\bullet C_0$ whose grade piece
$Gr_W^iC_0=W^iC_0/W^{i-1}C_0$ is pure of weight $i$. A morphism $C_0
\to C'_0$ of mixed perverse sheaves maps $W^iC_0$ to $W^iC'_0$.

\item
A mixed complex $C_0$ is of weight $\le w$ (resp. $\ge w$) if and
only if $\pH^i(C_0)$ is of weight $\le w+i$ (resp. $\ge w+i$) for
all $i$.
\end{enumerate}

\smallskip

Now we introduce the notion of pure resolution of mixed complexes.
Suppose we are given a system of exact triangles
\begin{equation}\label{eqn:mix1}
  C_0^{\le w-1} \to C_0^{\le w} \to C_0^w
\end{equation}
with $C_0^{\le w}\in\D_{\le w}(X_0)$ and $C_0^w$ being pure of
weight $w$. Assume further that the exact triangles \eqref{eqn:mix1}
are identical to $0\to0\to0$ for $w\ll0$ and are identical to
$C_0\xrightarrow{\Id} C_0\to0$ for $w\gg0$, where $C_0$ is a
prescribed mixed complex. Following Lemma \ref{lem:homology5} we
have a complex forming by objects and morphisms from $\D_m(X_0)$
\begin{equation}\label{eqn:mix2}
  \cdots \to C_0^{w+1}[-w-1] \to C_0^w[-w] \to C_0^{w-1}[-w+1] \to \cdots
\end{equation}
in which the differentials are the compositions
\begin{equation*}
  C_0^w[-w] \to C_0^{\le w-1}[-w+1] \to C_0^{w-1}[-w+1].
\end{equation*}

\begin{defn}
In the above notations, we assign degree $-w$ to $C_0^{w}[-w]$ and
call \eqref{eqn:mix2} a pure resolution of $C_0$.
\end{defn}

\smallskip

In particular, given a mixed perverse sheaf $C_0\in\M_m(X_0)$, the
unique weight filtration $W^\bullet C_0$ gives rise to a system of
exact sequences in $\M_m(X_0)$ (hence exact triangles in
$\D_m(X_0)$)
\begin{equation}
  W^{w-1}C_0 \hookrightarrow W^wC_0 \twoheadrightarrow Gr_W^wC_0,
\end{equation}
then a pure resolution of $C_0$
\begin{equation}\label{eqn:mix5}
  \cdots \to Gr_W^{w+1}C_0[-w-1] \to Gr_W^wC_0[-w] \to Gr_W^{w-1}C_0[-w+1] \to \cdots.
\end{equation}

\begin{defn}
We call \eqref{eqn:mix5} the canonical pure resolution of the mixed
perverse sheaf $C_0$. Moreover, when $C_0$ is clear from context, we
slightly abuse language to call the pullback of \eqref{eqn:mix5} the
canonical pure resolution of the perverse sheaf $C$.
\end{defn}

\begin{prop}\label{prop:mix}
We have the followings.
\begin{enumerate}
\setlength{\itemsep}{.3ex}
\item
If $f: X_0\to Y_0$ is a proper morphism, then $f_!$ transforms a
pure resolution of $C_0\in\D_m(X_0)$ into a pure resolution of
$f_!C_0\in\D_m(Y_0)$.

\item
If $f: X_0\to Y_0$ is a smooth morphism of relative dimension $d$,
then $f^*[d]$ transforms a pure resolution of $C_0\in\D_m(Y_0)$ into a pure
resolution of $f^*[d]C_0\in\D_m(X_0)$, up to a grade shifting.

\item
The Verdier duality functor $D$ transforms a pure resolution of
$C_0\in\D_m(X_0)$ into a pure resolution of $DC_0\in\D_m(X_0)$.

\item
The outer tensor product functor $\boxtimes$ transforms the
canonical pure resolutions of $C_0\in\M_m(X_0)$, $C'_0\in\M_m(Y_0)$
into the canonical pure resolution of $C_0\boxtimes
C'_0\in\M_m(X_0\times Y_0)$.
\end{enumerate}
Moreover, for a mixed perverse sheaf $C_0$, its canonical pure
resolution is transformed into canonical ones in (2) and
(3).
\end{prop}

\begin{proof}
Claim (1)(2) follow directly from \ref{sec:R:mix}.(3)(4).

We show then the third claim. Suppose $C_0^\bullet$ is a pure resolution of $C_0$
derived from a system of exact triangles
$$C_0^{\le w-1} \to C_0^{\le w} \to C_0^w.$$
We form triangle morphisms
$$\xymatrixrowsep{1.5pc}\xymatrix{
  B_0^{\le w-1} \ar[r] \ar[d] & C_0^{\le w-1} \ar[r]^---{d_{w-1}} \ar[d] & C_0 \ar@{=}[d] \\
  B_0^{\le w} \ar[r] & C_0^{\le w} \ar[r]^---{d_w} & C_0
}
$$
such that $d_w$ is the identity for large $w$.
By Lemma \ref{lem:homology6} there are exact triangles
$$B_0^{\le w-1} \to B_0^{\le w} \to C_0^w$$
inducing the given complex $C_0^\bullet$.
Then the exact triangles
\begin{equation}\label{eqn:mix1a}
  (DB_0^{\le w})[-1] \to (DB_0^{\le w-1})[-1] \to DC_0^w
\end{equation}
induce the Verdier dual of $C_0^\bullet$.

Using \ref{sec:R:mix}.(9)-(11) we can show by induction that
$\Img\pH^i(d_w) = W^{w+i}\,\pH^i(C_0)$ and that $\Ker\pH^i(d_w)$
is pure of weight $w+i$. Therefore, $\pH^i(B_0^{\le w})$ is of weight $\ge w+i$.
Hence $B_0^{\le w} \in \D_{\ge w}(X_0)$; $(DB_0^{\le w})[-1] \in \D_{\le -w-1}(X_0)$.
That being said, the exact triangles \eqref{eqn:mix1a} define a pure resolution
of $DC_0$. This proves Claim (3).

\smallskip

Below we prove Claim (4) by using \ref{sec:R:mix}.(5). First, we show
\begin{equation}\label{eqn:mix0a}
  Gr_W^\bullet(C_0\boxtimes C'_0) = \oplus_j Gr_W^jC_0\boxtimes Gr_W^{\bullet-j}C'_0.
\end{equation}
Assume $C_0$ is of weight $\le i$. Then by \ref{sec:R:mix}.(5) we
have exact sequences
\begin{equation}\label{eqn:mix0b}
  W^\bullet(W^{i-1}C_0\boxtimes C'_0) \hookrightarrow
  W^\bullet(C_0\boxtimes C'_0) \twoheadrightarrow
  W^\bullet(Gr_W^iC_0\boxtimes C'_0)
\end{equation}
and thus
\begin{equation*}
  Gr_W^\bullet(W^{i-1}C_0\boxtimes C'_0) \hookrightarrow
  Gr_W^\bullet(C_0\boxtimes C'_0) \twoheadrightarrow
  Gr_W^\bullet(Gr_W^iC_0\boxtimes C'_0).
\end{equation*}
Clearly
$$Gr_W^\bullet(Gr_W^iC_0\boxtimes C'_0) = Gr_W^iC_0\boxtimes Gr_W^{\bullet-i}C'_0.$$
By induction on weight we may assume further
$$Gr_W^\bullet(W^{i-1}C_0\boxtimes C'_0) = \oplus_{j<i} Gr_W^jC_0\boxtimes Gr_W^{\bullet-j}C'_0.$$
Moreover, from the K\"unneth formula
$$\Ext^\bullet_{\D(X_0\times Y_0)} (A_0\boxtimes B_0,A'_0\boxtimes B'_0)
  = \Ext^\bullet_{\D(X_0)} (A_0,A'_0) \otimes \Ext^\bullet_{\D(Y_0)} (B_0,B'_0)
$$
and \ref{sec:R:mix}.(7) we deduce that
$$\Ext^1_{\D(X_0\times Y_0)} (Gr_W^iC_0\boxtimes Gr_W^{\bullet-i}C'_0,\oplus_{j<i} Gr_W^jC_0\boxtimes Gr_W^{\bullet-j}C'_0) = 0.$$
Hence \eqref{eqn:mix0a} follows.

Next, we determine the differentials in the canonical pure
resolution of $C_0\boxtimes C'_0$
\begin{equation*}\label{eqn:mix0c}
  Gr_W^w(C_0\boxtimes C'_0)[-w] \to Gr_W^{w-1}(C_0\boxtimes C'_0)[-w+1].
\end{equation*}
By the the K\"unneth formula and \ref{sec:R:mix}.(7) again, the
morphisms by restriction
\begin{equation}\label{eqn:mix0d}
  Gr_W^jC_0\boxtimes Gr_W^{w-j}[-w] \to Gr_W^kC_0\boxtimes Gr_W^{w-1-k}C'_0[-w+1]
\end{equation}
must be zero unless $j=k$ or $k+1$. For $j=k$, by using the exact
sequences \eqref{eqn:mix0b} and by induction on weight, one verifies
that the induced morphisms \eqref{eqn:mix0d} coincide with the
differentials in the canonical pure resolutions of
$Gr_W^jC_0\boxtimes C'_0$. Similarly for $j=k+1$. This concludes
Claim (4).

The moreover part of the proposition is clear.
\end{proof}

The rest of this subsection is dedicated to the proof of the following theorem.

\begin{thm}\label{thm:mix}
Let $C_0\in\M_m(X_0)$ be a mixed perverse sheaf. Then the pullbacks
(to $X$) of the pure resolutions of $C_0$ are all homotopy
equivalent.
\end{thm}

\begin{proof}
Suppose we are given a pure resolution of $C_0$ derived from a
system of exact triangles
\begin{equation}\label{eqn:unid}
  C_0^{\le w-1} \to C_0^{\le w} \to C_0^w.
\end{equation}
We shall show that its pullback to $X$ is homotopy equivalent to the
pullback of the canonical one.

Let $\ptau_{\le n}, \ptau_{\ge n}$ denote the truncation functors of
$\D(X_0)$ for the perverse t-structure.
If all $C_0^w$ are mixed perverse sheaves, an easy induction shows
that so are $C_0^{\le w}$. Property \ref{sec:R:mix}.(9)(10) then imply in this case
that the exact triangles \eqref{eqn:unid} must be the canonical ones
$W^{w-1}C_0 \to W^wC_0 \to Gr_W^wC_0$; we are done.
Otherwise, there exists $C_0^w$ such that $\ptau_{\le-1}C_0^w\neq0$
or $\ptau_{\ge1}C_0^w\neq0$. Without loss of generality we consider
the former case; the latter can be treated by passing to Verdier dual.

Let $w$ be maximal such that $\ptau_{\le-1}C_0^w\neq0$.
We form a triangle morphism
\begin{equation}\label{eqn:unie}
\xymatrixrowsep{1.5pc}\xymatrix{
  C_0^{\le w-1} \ar[d]_e \ar[r] & C_0^{\le w} \ar@{=}[d] \ar[r] & C_0^w \ar[d]^p \\
  \tilde{C}_0^{\le w-1} \ar[r] & C_0^{\le w} \ar[r] & \ptau_{\ge0}C_0^w
}
\end{equation}
where $p$ is the morphism in the exact triangle
$$\ptau_{\le-1}C_0^w \to C_0^{w} \xrightarrow{p} \ptau_{\ge0}C_0^w.$$
Applying Lemma \ref{lem:homology4} to the sequence
$$C_0^{\le w-2} \to C_0^{\le w-1} \xrightarrow{e} \tilde{C}_0^{\le w-1} \to C_0^{\le w},$$
then gives an exact triangle
\begin{equation}\label{eqn:unia}
  C_0^{\le w-2} \to \tilde{C}_0^{\le w-1} \to \tilde{C}_0^{w-1},
\end{equation}
a triangle morphism
\begin{equation}\label{eqn:unib}
\xymatrixrowsep{1.5pc}\xymatrix{
  \ptau_{\le-1}C_0^w \ar@{=}[d] \ar[r] & C_0^{w} \ar[d] \ar[r]^p & \ptau_{\ge0}C_0^w \ar[d] \ar[r]^----u & \\
  \ptau_{\le-1}C_0^w \ar[r] & C_0^{w-1}[1] \ar[r] & \tilde{C}_0^{w-1}[1] \ar[r]^----v &
}
\end{equation}
and a complex morphism
\begin{equation}\label{eqn:unic}
\xymatrixrowsep{1.5pc}\xymatrixcolsep{1pc}\xymatrix{
  \cdots \ar[r] & C_0^{w+1}[-1] \ar[r] \ar@{=}[d] & C_0^w \ar[r] \ar[d]
  & C_0^{w-1}[1] \ar[r] \ar[d] & C_0^{w-2}[2] \ar[r] \ar@{=}[d] & \cdots \\
  \cdots \ar[r] & C_0^{w+1}[-1] \ar[r] & \ptau_{\ge0}C_0^w \ar[r]
  & \tilde{C}_0^{w-1}[1] \ar[r] & C_0^{w-2}[2] \ar[r] & \cdots
}
\end{equation}
of which the bottom row is the one derived from the exact triangles
\eqref{eqn:unid} with $C_0^{w-1}$, $C_0^{\le w-1}$, $C_0^{w}$
replaced by $\tilde{C}_0^{w-1}$, $\tilde{C}_0^{\le w-1}$, $\ptau_{\ge0}C_0^w$,
respectively.

By the maximality of $w$, an inductive argument shows
$\ptau_{\le-1}C_0^{\le w}=0$.
Then from \ref{sec:R:mix}.(9)-(11) and from the long exact sequences associated to
the exact triangles in \eqref{eqn:unie}, we deduce that
\begin{equation*}
\pH^i(\tilde{C}_0^{\le w-1}) =
\begin{cases}
  \pH^i(C_0^{\le w-1}), & i>0, \\
  W^{w-1}\,\pH^0(C_0^{\le w}), & i=0, \\
  0, & i<0.
\end{cases}
\end{equation*}
Hence $\tilde{C}_0^{\le w-1}$ is of weight $\le w-1$.
Further, from \eqref{eqn:unia} and the bottom row of \eqref{eqn:unib}
we conclude that $\tilde{C}_0^{w-1}$ is pure of weight $w-1$.

It follows on the one hand that, up to a grade
shifting, the bottom row of \eqref{eqn:unic} is a pure resolution of
$C_0$; and on the other hand that the pullbacks of $u,v$ to $X$ are
zero by \ref{sec:R:mix}.(8), thus by Lemma \ref{lem:homology2} the
pullback of \eqref{eqn:unic} to $X$ is a homotopy equivalence.
Summarizing, we obtain a new pure resolution of $C_0$ whose pullback
to $X$ is homotopy equivalent to that of the original one.

Note that the above process remains all $C^i$ untouched
but truncates off nontrivial direct summands from $C^{w-1}$ and $C^w$.
Therefore, after finitely many repetitions, the original pure resolution can be
deformed to the canonical one. This completes the proof of our
theorem.
\end{proof}

\begin{rem}
The claim of the theorem may not be true if we do not pull back
pure resolutions to $X$. For example, let $X_0 = Spec(\BF_q)$ and,
accordingly, $X = Spec(\BF)$. We can form an exact sequence of pure
perverse sheaves (of weight 0)
$$\cdots \to 0 \to \bar\Q_{l,X_0} \to A_0 \to \bar\Q_{l,X_0} \to 0 \to \cdots$$
with $A_0$ indecomposable \cite[5.3.9]{BBD82}. It is easy to realize
the above sequence as a pure resolution of the zero mixed perverse
sheaf, which is, however, may not be homotopic to zero. Indeed, the
existence of such indecomposable $A_0$ is the obstacle preventing
the morphisms $u,v$ in \eqref{eqn:unib} from being zero. If we pull
back the above sequence to $X$, it yields now a complex homotopic to
zero
$$\cdots \to 0 \to \bar\Q_{l,X} \to \bar\Q_{l,X}\oplus\bar\Q_{l,X} \to \bar\Q_{l,X} \to 0 \to \cdots.$$
\end{rem}

\subsection{The complex $T^\bullet$}\label{sec:R:R}

First, let us remark that, according to the standard reduction
technique \cite[6.1]{BBD82} from the base field $\C$ to finite
fields, it makes sense to pull back a mixed complex to a complex
algebraic variety $X$ thus form a complex of $\C$-sheaves, as if $X$
is obtained from a scheme over a finite field by base field
extension.

Readers who are unsatisfactory with such reduction may simply bypass
it by transferring from the very beginning of this paper to the
setting of varieties over algebraic closures of finite fields and
categories of $\bar\Q_l$-sheaves.

\smallskip

Keep the notations of Section \ref{sec:pre:flag}. Let $X=G/P$ be a
partial flag variety. For each $w\in\W^P$ we set
\begin{equation}
  \Delta^+_w = j_{w!}\C_{X_w}[\dim X_w], \quad \Delta^-_w = D\Delta^+_w
\end{equation}
where $j_w: X_w\to X$ is the inclusion. Since $j_w$ is an affine
morphism, $\Delta^\pm_w$ are perverse sheaves on $X$.

By regarding $\C_{X_w}$ as the pullback of a constant
$\bar\Q_l$-sheaf (pure of weight 0) for each $w\in\W^P$, we are
clear from which mixed perverse sheaves $\Delta^\pm_w$ are pulled
back. Then we define $T^\bullet(P,\Delta^\pm_w)$ to be the canonical
pure resolutions of $\Delta^\pm_w\in\M_B(X)$.

The first properties of these complexes are as follows.

\begin{prop}
Let $P\subset G$ be parabolic subgroups containing $B$.
\begin{enumerate}
\setlength{\itemsep}{.3ex}
\item
$T^n(P,\Delta^\pm_w)[-n]$ are semisimple subquotients of
$\Delta^\pm_w$. In particular, they are self dual and
$B$-equivariant.

\item
$DT^\bullet(P,\Delta^\pm_w) = T^\bullet(P,\Delta^\mp_w)$.
\end{enumerate}
\end{prop}

\begin{proof}
Claim (1) is immediate from the definition of canonical pure
resolution and \ref{sec:R:mix}.(2). Claim (2) follows from
Proposition \ref{prop:mix}(3).
\end{proof}

\begin{exam}
For the unit element $e\in\W^P$, $T^\bullet(P,\Delta^\pm_e)$ are
clearly the complex concentrated at degree $0$
$$\cdots \to 0 \to 0 \to \Delta^+_e \to 0 \to 0 \to \cdots.$$
\end{exam}

\begin{exam}\label{exam:R:2}
For a simple reflection $s\in\W^P$, note that $\overline{X}_{s} =
X_{s} \sqcup X_e \cong \BP^1$ and $\IC(\overline{X}_{s}) =
\C_{\overline{X}_{s}}[1]$. The complex $T^\bullet(P,\Delta^+_s)$ is
the one concentrated at degree $-1,0$
$$\cdots \to 0 \to \IC(\overline{X}_{s})[-1] \xrightarrow{a} \Delta^+_e \to 0 \to 0 \to \cdots$$
where $a$ is the adjunction morphism $\C_{\overline{X}_{s}} \to
j_{e*}j_e^*\C_{\overline{X}_{s}}$. More precisely,
$T^\bullet(P,\Delta^+_s)$ is the one derived from the exact
sequences
\begin{align*}
  & 0 \hookrightarrow \Delta^+_e \twoheadrightarrow \Delta^+_e, \\
  & \Delta^+_e \hookrightarrow \Delta^+_{s} \twoheadrightarrow \IC(\overline{X}_{s}),
\end{align*}
of which the latter is actually the adjunction triangle
$$j_{s!}j_{s}^!\IC(\overline{X}_{s})
  \to \IC(\overline{X}_{s})
  \to j_{e*}j_e^*\IC(\overline{X}_{s}).
$$
\end{exam}

\begin{exam}
In the notations of Section \ref{sec:mod:Q}, we have the followings,
where $\IC(\overline{X}_w)$ is abbreviated to $\IC_w$.
\begin{footnotesize}
\begin{equation*}
\begin{array}{l @{\;=\;\cdots\;\to\;} *{3}{c@{\;\to\;}} c @{\;\to\;\cdots}}
  \vspace{3pt}
  T^\bullet(P_{(1,2)},\Delta^+_{s_2s_1}) & 0 & 0 & \IC_{s_2s_1}[-2] & \IC_{s_1}[-1] \\
  \vspace{3pt}
  T^\bullet(P_{(1,3)},\Delta^+_{s_3s_2s_1}) & 0 & \IC_{s_3s_2s_1}[-3] & \IC_{s_2s_1}[-2] & 0 \\
  T^\bullet(P_{(2,2)},\Delta^+_{s_2s_1s_3s_2}) & \IC_{s_2s_1s_3s_2}[-4] & \IC_{s_1s_3s_2}[-3] & \IC_e[-2] & 0 \\
\end{array}
\end{equation*}
\end{footnotesize}
\end{exam}

For a collection of parabolic subgroups $P,P_1,\dots,P_k\subset G$
containing $B$, there is a principal $P$-bundle
\begin{align*}
  \mu^{P,P_1,\dots,P_k} : G\times G/P_1\times\cdots\times G/P_k
  & \; \to \; G/P\times G/P_1\times\cdots\times G/P_k \\
  (g,x_1,\dots,x_k) & \; \mapsto \; ([g],gx_1,\dots,gx_k).
\end{align*}
Recall that given a principal $P$-bundle $\mu: X\to Y$, the functor
$\mu^*[\dim P]$ is perverse t-exact and, moreover, together with the
functor $\mu_\flat=\pH^{-\dim P}\mu_*$ it defines an equivalence of
the categories $\M_P(X)$, $\M(Y)$. By abusing notations, when
$T^\bullet$ is the complex derived from a system of exact sequences
$C^{\le w-1} \hookrightarrow C^{\le w} \twoheadrightarrow C^w$ in
$\M_P(X)$ (cf. Lemma \ref{lem:homology5}), we denote by $\mu_\flat
T^\bullet$ the complex derived from the system of exact sequences
$\mu_\flat C^{\le w-1} \hookrightarrow \mu_\flat C^{\le w}
\twoheadrightarrow \mu_\flat C^w$ in $\M(Y)$.

\smallskip

Let $P,Q\subset G$ be parabolic subgroups containing $B$ and let $G$
acts on $G/P\times G/Q$ diagonally. Let $\W^{P,Q}$ be the set of
shortest representatives of the double cosets
$\W^P\backslash\W/\W^Q$. Then we have a decomposition by $G$-orbits
\begin{equation}
  G/P\times G/Q=\bigsqcup_{w\in\W^{P,Q}} O_w
\end{equation}
where $O_w$ is the $G$-orbit of $(P,\dot{w}Q)$.

Notice the one-to-one correspondence between the $G$-orbits of
$G/P\times G/Q$ and the $P$-orbits of $G/Q$
\begin{equation}\label{eqn:R:orbit}
  O_w \leftrightarrow P\dot{w}Q/Q.
\end{equation}

\smallskip

Assume $w\in\W^{P,Q}$ is such that $w_Pw=ww_Q$ where $w_P,w_Q$ are
the longest elements in $\W_P,\W_Q$, respectively. Then
$\Delta^\pm_w \in \M_P(G/Q)$. We define
$T^\bullet(P,Q,\Delta^\pm_w)$ to be the canonical pure resolutions
of
$$\mu^{P,Q}_\flat(\C_G[\dim G]\boxtimes\Delta^\pm_w) \in \M_G(G/P\times G/Q).$$
The following proposition follows easily from \ref{sec:R:mix}.(2)
and Proposition \ref{prop:mix}(2).

\begin{prop}\label{prop:R0b}
Let $P,Q\subset G$ be parabolic subgroups containing $B$.
\begin{enumerate}
\setlength{\itemsep}{.3ex}
\item
$T^n(P,Q,\Delta^\pm_w)[-n]$ are semisimple $G$-equivariant perverse
sheaves.

\item
$\tau^*T^\bullet(P,Q,\Delta^\pm_w) =
T^\bullet(Q,P,\Delta^\pm_{w^{-1}})$, where $\tau: G/P\times G/Q \to
G/Q\times G/P$ is the transposition.

\item
$T^{\bullet-\dim G/P} (P,Q,\Delta^\pm_w) = \mu^{P,Q}_\flat \big(
\C_G[\dim G]\boxtimes T^\bullet(Q,\Delta^\pm_w) \big) [-\dim G/P].$
\end{enumerate}
\end{prop}

Next, let $X,Y,Z$ be algebraic varieties and recall the convolution
product of $C\in\D(X\times Y)$ and $C'\in\D(Y\times Z)$
\begin{equation}
  C*C' = \pi_{13!} (\pi_{12}^*C \otimes \pi_{23}^*C') \in\D(X\times Z)
\end{equation}
where $\pi_{ij}$ is the projection of $X\times Y\times Z$ onto the
$i,j$-th coordinates. In this way, each $C\in\D(X\times Y)$ gives
rise to a functor
\begin{equation}
  C*: \D(Y\times Z) \to \D(X\times Z).
\end{equation}
It is left adjoint to the functor
\begin{equation}
  D\circ\tau^*C*\circ D : \D(X\times Z) \to \D(Y\times Z)
\end{equation}
where $\tau: X\times Y\to Y\times X$ is the transposition.

\begin{lem}\label{lem:R2}
Let $P,Q,R\subset G$ be parabolic subgroups containing $B$. There
are natural isomorphisms for $C_1\in\M_P(G/Q)$, $C_2\in\M_Q(G/R)$,
\begin{align*}
  & \pi_{12}^* \mu^{P,Q}_\flat \big( \C_G[\dim G]\boxtimes C_1 \big) [-\dim G/P] \otimes
  \pi_{23}^* \mu^{Q,R}_\flat \big( \C_G[\dim G]\boxtimes C_2 \big) [-\dim G/Q] \\
  & \cong \mu^{P,Q,R}_\flat \Big( \C_G[\dim G]\boxtimes
  \mu^{Q,R}_\flat \big( \mu^{Q*}[\dim Q]C_1\boxtimes C_2 \big) \Big) [-\dim G/P].
\end{align*}
\end{lem}

\begin{proof}
A direct computation.
\end{proof}

\begin{lem}\label{lem:R3}
Let $P,Q,R\subset G$ be parabolic subgroups containing $B$. The
complex
\begin{equation}\label{eqn:R:1}
  \pi_{12}^* T^\bullet(P,Q,\Delta^\epsilon_{w_1})
  \otimes \pi_{23}^* T^{\bullet-\dim G/Q}(Q,R,\Delta^\varepsilon_{w_2})
\end{equation}
is the canonical pure resolution of the perverse sheaf on $G/P\times
G/Q\times G/R$
\begin{equation}\label{eqn:R:2}
  \mu^{P,Q,R}_\flat \Big( \C_G[\dim G] \boxtimes
  \mu^{Q,R}_\flat \big( \mu^{Q*}[\dim Q]\Delta^\epsilon_{w_1} \boxtimes \Delta^\varepsilon_{w_2} \big) \Big).
\end{equation}
\end{lem}

\begin{proof}
By Proposition \ref{prop:mix}(2)(4), the canonical pure resolution
of
$$\mu^{Q,R}_\flat (\mu^{Q*}[\dim Q]\Delta^\epsilon_{w_1} \boxtimes \Delta^\varepsilon_{w_2})$$
is
$$\mu^{Q,R}_\flat \Big( \mu^{Q*}[\dim Q] T^\bullet(Q,\Delta^\epsilon_{w_1}) \boxtimes
  T^\bullet(R,\Delta^\varepsilon_{w_2}) \Big).
$$
Thus, the canonical pure resolution of \eqref{eqn:R:2} is
$$\mu^{P,Q,R}_\flat \Big( \C_G[\dim G] \boxtimes
  \mu^{Q,R}_\flat \Big( \mu^{Q*}[\dim Q] T^{\bullet+\dim G/P}(Q,\Delta^\epsilon_{w_1}) \boxtimes
  T^\bullet(R,\Delta^\varepsilon_{w_2}) \Big) \Big) [-\dim G/P]
$$
which by Proposition \ref{prop:R0b}(3) and Lemma \ref{lem:R2} is
equal to \eqref{eqn:R:1}.
\end{proof}

\begin{prop}\label{prop:R1}
Let $P,Q,R\subset G$ be parabolic subgroups containing $B$. Then
$$T^\bullet(P,Q,\Delta^\pm_{w_1}) * T^{\bullet-\dim G/Q}(Q,R,\Delta^\pm_{w_2})
  \simeq T^\bullet(P,R,\Delta^\pm_{w_1w_2})
$$
whenever $\ell(w_1w_2) = \ell(w_1) + \ell(w_2)$.
\end{prop}

\begin{proof}
Suppose $\ell(w_1w_2) = \ell(w_1) + \ell(w_2)$. Thus we have
$$\pi_{2!} \mu^{Q,R}_\flat \big( \mu^{Q*}[\dim Q]\Delta^\pm_{w_1} \boxtimes \Delta^\pm_{w_2} \big)
  \cong \Delta^\pm_{w_1w_2},
$$
where $\pi_2: G/Q\times G/R\to G/R$ is the projection onto the
second coordinate. It follows that
\begin{align*}
  & \pi_{13!} \mu^{P,Q,R}_\flat \Big( \C_G[\dim G] \boxtimes
  \mu^{Q,R}_\flat \big( \mu^{Q*}[\dim Q]\Delta^\pm_{w_1} \boxtimes
  \Delta^\pm_{w_2} \big) \Big) \\
  & \cong \mu^{P,R}_\flat \big( \C_G[\dim G] \boxtimes \Delta^\pm_{w_1w_2} \big).
\end{align*}
Then apply Lemma \ref{lem:R3} and Proposition \ref{prop:mix}(1),
Theorem \ref{thm:mix}.
\end{proof}

The next lemma remains true if we change the base field $\C$ to
finite fields. To make the point clear, we write down the Tate twist
terms in its proof.

\begin{lem}\label{lem:R4}
For each simple reflection $s\in\W$, we have
$$\pi_{2!}\mu^{B,B}_\flat \big( \mu^{B*}[\dim B]\Delta^\pm_s \boxtimes \Delta^\mp_s \big) \cong \Delta^\pm_e$$
where $\pi_2: G/B\times G/B\to G/B$ is the projection onto the
second coordinate.
\end{lem}

\begin{proof}
We put $\Delta = \mu^{B,B}_\flat \big( \mu^{B*}[\dim B]\Delta^+_s
\boxtimes \Delta^-_s \big)$. Then
$$\Delta
  = \mu^{B,B}_\flat \big( \mu^{B*}[\dim B]j_{s!}\C_{X_s}[1] \boxtimes j_{s*}\C_{X_s}[1](1) \big)
  = j_{Z_1!}j_{Z_0*}\C_{Z_0}[2](1)
$$
where $j_{Z_0}: Z_0\to Z_1$, $j_{Z_1}: Z_1\to G/B\times G/B$ are the
inclusions of subvarieties
\begin{align*}
  & Z_0 = \{ ([g_1],[g_1g_2]) \mid g_1\in B\dot{s}B, \, g_2\in B\dot{s}B \} \subset G/B\times G/B, \\
  & Z_1 = \{ ([g_1],[g_1g_2]) \mid g_1\in B\dot{s}B, \, g_2\in \overline{B\dot{s}B} \} \subset G/B\times G/B.
\end{align*}
Set $Y=Z_1\setminus Z_0$ and let $i_Y: Y\to Z_1$ be the inclusion.
Note that via the morphism $\pi_2$, $Z_1$ becomes a $\C$-bundle over
$\overline{X}_s\cong\BP^1$. The functor $\pi_{2!}j_{Z_1!}$
transforms the adjunction triangle
$$i_{Y!}\C_Y \to \C_{Z_1}[2](1) \to j_{Z_0*}\C_{Z_0}[2](1)$$
into an exact triangle
$$j_{s!}\C_{X_s} \to \C_{\overline{X}_s} \to \pi_{2!}\Delta.$$
Applying the functors $j_e^*$, $j_s^*$, one verifies that
$$j_e^*\pi_{2!}\Delta\cong j_e^*\Delta^+_e, \quad j_s^*\pi_{2!}\Delta\cong0.$$
Hence the isomorphism $\pi_{2!}\Delta \cong \Delta^+_e$ follows.

Applying the Verdier duality functor $D$ yields the other
isomorphism.
\end{proof}

\begin{prop}\label{prop:R2}
Let $P,Q\subset G$ be parabolic subgroups containing $B$. Then
$$T^\bullet(P,Q,\Delta^\pm_{w^{-1}}) * T^{\bullet-\dim G/Q}(Q,P,\Delta^\mp_w)
  \simeq T^\bullet(P,P,\Delta^\pm_e).
$$
\end{prop}

\begin{proof}
Choose a reduced word $w=s_{i_1}s_{i_2}\cdots s_{i_t}$. Defines a
principal $B^{2t}$-bundle
\begin{align*}
  \mu: G^{2t} & \; \to \; (G/B)^{2t} \\
  (g_1,g_2,\dots,g_{2t}) & \; \mapsto \; ([g_1],[g_1g_2],\dots,[g_1\cdots g_{2t}]).
\end{align*}
Let $\pi_2: G/Q\times G/P\to G/P$ and $\pi_{2t}: (G/B)^{2t}\to G/B$
be the projections onto the last coordinates and let $\rho: G/B\to
G/P$ be the obvious projection. By a direct computation,
\begin{align*}
  & \pi_{2!}\mu^{Q,P}_\flat \Big( \mu^{Q*}[\dim Q]\Delta^\pm_{w^{-1}} \boxtimes \Delta^\mp_w \Big) \\
  \cong \;& \rho_!\pi_{2t!}\mu_\flat \Big( \mu^{B*}\Delta^\pm_{s_{i_t}} \boxtimes
  \cdots \boxtimes \mu^{B*}\Delta^\pm_{s_{i_1}}
  \boxtimes \mu^{B*}\Delta^\mp_{s_{i_1}} \boxtimes
  \cdots \boxtimes \mu^{B*}\Delta^\mp_{s_{i_t}} [2t\dim B] \Big).
\end{align*}
Applying Lemma \ref{lem:R4} repeatedly, we deduce that the right
hand side is further isomorphic to $\Delta^\pm_e$.

Then, the same argument as the proof of Proposition \ref{prop:R1}
concludes the proposition.
\end{proof}

\begin{prop}\label{prop:R3}
Let $P,Q,R_1,R_2\subset G$ be parabolic subgroups containing $B$.
Assume $R_1\subset R_2$ and let $p: G/R_1\to G/R_2$ be the obvious
projection. We have natural isomorphisms for $C\in\D(G/P\times G/Q)$
and $C_i\in\D(G/Q\times G/R_i)$
\begin{align*}
  (\Id_{G/P}\times p)_!(C*C_1) &= C*(\Id_{G/Q}\times p)_!C_1, \\
  (\Id_{G/P}\times p)^*(C*C_2) &= C*(\Id_{G/Q}\times p)^*C_2.
\end{align*}
\end{prop}

\begin{proof}
An easy computation.
\end{proof}

\subsection{Categorification theorem}\label{sec:R:cat}

Now we transfer to the notations of Section \ref{sec:mod:Q}. Let
$X=\sqcup_r X_d^r$ be the union of Grassmannian varieties.

First, we enhance the category $\CQ_\bd$ to $\tilde\CQ_\bd$, which
is defined to be the full subcategory of $\D(G/P_\bd\times X)$
consisting of the $G$-equivariant semisimple complexes. Notice the
canonical correspondence between the $G$-orbits of $G/P_\bd\times X$
and the $P_\bd$-orbits of $X$ given by \eqref{eqn:R:orbit}. The key
observation is that the functor
\begin{equation}
  \mu^{P_\bd,\cdot}_\flat (\C_G[\dim G]\boxtimes -): \M_{P_\bd}(X) \to \M_G(G/P_\bd\times X)
\end{equation}
gives rise to a one-to-one correspondence between the (isomorphism
classes of) perverse sheaves from $\CQ_\bd$ and $\tilde\CQ_\bd$. In
particular, it identifies the Grothendieck group of $\tilde\CQ_\bd$
with $Q_\bd$.

All concepts and claims from Section \ref{sec:mod} can be migrated
word by word to a version for $\tilde\CQ_\bd$ in the obvious way.
For example, $\K_\bd, \E^{(n)}_\bd, \F^{(n)}_\bd$ are endofunctors
of $\D(G/P_\bd \times X)$ defined by assembling
\begin{equation}
\begin{array}{lcl}
  \K_{\bd,r} = [2r-d] & : & \D(G/P_\bd\times X_d^r) \to \D(G/P_\bd\times X_d^r), \\
  \E^{(n)}_{\bd,r+n} = p_!p'^*[nr] & : & \D(G/P_\bd\times X_d^{r+n}) \to \D(G/P_\bd\times X_d^r), \\
  \F^{(n)}_{\bd,r} = p'_!p^*[n(d-n-r)] & : & \D(G/P_\bd\times X_d^r) \to \D(G/P_\bd\times X_d^{r+n}),
\end{array}
\end{equation}
where $p,p'$ are the projections
\begin{equation}
\xymatrix@1{
  G/P_\bd\times X_d^r & \ar[l]_{p} G/P_\bd\times X_d^{r,r+n} \ar[r]^-----{p'} & G/P_\bd\times X_d^{r+n}
}
\end{equation}
They satisfy the functor isomorphisms stated in the propositions
from Section \ref{sec:mod:cat}, hence induce the same $U_\A$-module
structure on $Q_\bd$ as the functors $\K, \E^{(n)}, \F^{(n)}$.

\smallskip

Next, we identify the Weyl group $\W$ of $G=GL(W)$ with the
symmetric group $S_d$ of the symbols $\{1,2,\dots,d\}$. For each
composition $\bd=(d_1,d_2,\dots,d_l)$ of $d$ and for each
permutation $\sigma\in S_l$, we let $\sigma$ act on the symbols
$\{1,2,\dots,d\}$ by permuting the blocks
$$\{ d_1+\cdots+d_{i-1}+j \mid 1\le j\le d_i \}_{1\le i\le l}$$
thus yield an element $w(\bd,\sigma)\in\W$, then define a couple of
complexes formed by functors from $\D(G/P_\bd \times X)$ to
$\D(G/P_{\sigma(\bd)} \times X)$
\begin{equation}
  \CR^\bullet_\pm(\bd,\sigma) = D\circ  T^{\bullet-\dim G/P_\bd}
  \big( P_{\sigma(\bd)},P_\bd,\Delta^\mp_{w(\bd,\sigma)} \big) * \circ D.
\end{equation}
They are understood in the standard way as functors of bounded
homotopic categories
\begin{equation}
  \CR^\bullet_\pm(\bd,\sigma): \K^b(\D(G/P_\bd \times X))
  \to \K^b(\D(G/P_{\sigma(\bd)} \times X)).
\end{equation}

\begin{thm}\label{thm:catR}
The functor complexes $\CR^\bullet_\pm(\bd,\sigma)$ satisfy the
followings.
\begin{enumerate}
\setlength{\itemsep}{.3ex}
\item
$\CR^\bullet_\pm(\sigma_2(\bd),\sigma_1) \circ
\CR^\bullet_\pm(\bd,\sigma_2) \simeq
\CR^\bullet_\pm(\bd,\sigma_1\sigma_2)$ if $\ell(\sigma_1\sigma_2) =
\ell(\sigma_1) + \ell(\sigma_2)$.

\item
$\CR^\bullet_\pm(\sigma(\bd),\sigma^{-1}) \circ
\CR^\bullet_\mp(\bd,\sigma) \simeq \Id$ (concentrated at degree
$0$).

\item
$\G_{\sigma(\bd)} \circ \CR^\bullet_\pm(\bd,\sigma) \cong
\CR^\bullet_\pm(\bd,\sigma) \circ \G_\bd$ for $\G\in\{ \K, \E^{(n)},
\F^{(n)} \}$.

\item
$\CR^n_\pm(\bd,\sigma)$ are right adjoint to
$D\CR^n_\pm(\sigma(\bd),\sigma^{-1})D$.

\item
$\CR^n_\pm(\bd,\sigma)$ restrict to functors from $\tilde\CQ_\bd$ to
$\tilde\CQ_{\sigma(\bd)}$.
\end{enumerate}
\end{thm}

\begin{proof}
(1)(2)(3) follow from Proposition \ref{prop:R1}, \ref{prop:R2},
\ref{prop:R3}, respectively.

(4) follows from Proposition \ref{prop:R0b}(2).

(5) follows from Lemma \ref{lem:R2} and the decomposition theorem.
\end{proof}

\begin{cor}\label{cor:catR}
Let $\K^b(\tilde\CQ_\bd)$ be the bounded homotopic category of
$\tilde\CQ_\bd$. Via the valuation
$$\K^b(\tilde\CQ_\bd)\to Q_\bd, \quad C^\bullet\mapsto\sum_n(-1)^n[C^n],$$
the system of functors
$$\CR^\bullet_+(\bd,\sigma): \K^b(\tilde\CQ_\bd) \to \K^b(\tilde\CQ_{\sigma(\bd)})$$
induce a system of $R$-matrices on the $U_\A$-modules $Q_\bd$.
\end{cor}

\begin{proof}
It suffices to show the valuation only depends on the homotopy
equivalence class of $C^\bullet$. But this is evident from
Proposition \ref{prop:tAmof} below.
\end{proof}

\begin{exam}\label{exam:catR1}
For the simplest nontrivial case $\bd=(1,1)$, $S_2=\{e,\sigma\}$, it
is not difficult to derive from Example \ref{exam:R:2} and Lemma
\ref{lem:R2} that $\CR^\bullet_+(\bd,\sigma)$ (left column) and
$\CR^\bullet_-(\bd,\sigma)$ (right column) induce the following maps
$$\begin{array}{ll}
  b_{(0,0)} \mapsto -q^2b_{(0,0)}, & \quad b_{(0,0)} \mapsto -q^{-2}b_{(0,0)}, \\
  b_{(1,0)} \mapsto b_{(1,0)} - qb_{(0,1)}, & \quad b_{(1,0)} \mapsto b_{(1,0)} - q^{-1}b_{(0,1)}, \\
  b_{(0,1)} \mapsto -q^2b_{(0,1)}, & \quad b_{(0,1)} \mapsto -q^{-2}b_{(0,1)}, \\
  b_{(1,1)} \mapsto -q^2b_{(1,1)}, & \quad b_{(1,1)} \mapsto -q^{-2}b_{(1,1)}.
  \end{array}
$$
\end{exam}

In what follows we write $1^d = (1,1,\dots,1)$ for the composition
of $d$ consisting of the $1$'s. The category $\CQ_\bd$ is by
definition a full subcategory of $\CQ_{1^d}$, thus $Q_\bd$ is
canonically embedded in $Q_{1^d}$. Observe that the same embedding
is induced by the functor
$$(\rho_\bd\times\Id_X)^*[\dim P_\bd/B]: \tilde\CQ_{\bd}\to\tilde\CQ_{1^d}$$
where $\rho_\bd: G/B\to G/P_\bd$ is the obvious projection. The next
proposition states that $\CR^\bullet_\pm(\bd,\sigma)$ and
$\CR^\bullet_\pm(1^d,w(\bd,\sigma))$ induce the same isomorphisms
$Q_\bd \to Q_{\sigma(\bd)}$.

\begin{prop}\label{prop:catR}
Let $\rho_\bd: G/B\to G/P_\bd$ be the obvious projection. Then
$$(\rho_{\sigma(\bd)}\times\Id_X)^* \circ \CR^\bullet_\pm(\bd,\sigma)
  \simeq \CR^\bullet_\pm(1^d,w(\bd,\sigma)) \circ (\rho_\bd\times\Id_X)^*.
$$
\end{prop}

\begin{proof}
Put $w=w(\bd,\sigma)$. We have $\rho_{\bd!}\Delta^\pm_w =
\Delta^\pm_w$ ($\Delta^\pm_w$ defined respectively on $G/B$ and
$G/P_\bd$). Hence
\begin{align*}
  & (\Id_{G/B}\times\rho_\bd)_!
  \mu^{B,B}_\flat(\C_G[\dim G]\boxtimes\Delta^\mp_w) \\
  = \;& (\rho_{\sigma(\bd)}\times\Id_{G/P_\bd})^* [\dim P_\bd/B]
  \mu^{P_{\sigma(\bd)},P_\bd}_\flat(\C_G[\dim G]\boxtimes\Delta^\mp_w).
\end{align*}
Hence, by Proposition \ref{prop:mix}(1)(2) and Theorem
\ref{thm:mix},
\begin{align*}
  & (\Id_{G/B}\times\rho_\bd)_!
  T^{\bullet-\dim G/B} (B,B,\Delta^\mp_w) \\
  \simeq \;& (\rho_{\sigma(\bd)}\times\Id_{G/P_\bd})^*
  T^{\bullet-\dim G/P_\bd} (P_{\sigma(\bd)},P_\bd,\Delta^\mp_w).
\end{align*}
It follows that
\begin{align*}
  & \CR^\bullet_\pm(1^d,w(\bd,\sigma)) \circ (\rho_\bd\times\Id_X)^* \\
  = \;& D \pi_{13!} \Big( \pi_{12}^* T^{\bullet-\dim G/B} (B,B,\Delta^\mp_w) \otimes
  \pi_{23}^* (\rho_\bd\times\Id_X)^*[2\dim P_\bd/B] D- \Big) \\
  = \;& D \pi_{13!} \Big( \pi_{12}^* (\Id_{G/B}\times\rho_\bd)_! T^{\bullet-\dim G/B} (B,B,\Delta^\mp_w) \otimes
  \pi_{23}^* [2\dim P_\bd/B] D- \Big) \\
  \simeq \;& D \pi_{13!} \Big( \pi_{12}^* (\rho_{\sigma(\bd)}\times\Id_{G/P_\bd})^*
  T^{\bullet-\dim G/P_\bd} (P_{\sigma(\bd)},P_\bd,\Delta^\mp_w) \otimes
  \pi_{23}^* [2\dim P_\bd/B] D- \Big) \\
  = \;& (\rho_{\sigma(\bd)}\times\Id_X)^* D \pi_{13!}
  \Big( \pi_{12}^* T^{\bullet-\dim G/P_\bd} (P_{\sigma(\bd)},P_\bd,\Delta^\mp_w) \otimes \pi_{23}^* D- \Big) \\
  = \;& (\rho_{\sigma(\bd)}\times\Id_X)^* \circ \CR^\bullet_\pm(\bd,\sigma).
  \tag*\qed
\end{align*}
\renewcommand{\qed}{}
\end{proof}

\begin{exam}
Assume $\bd=(d_1,d_2)$, $S_2=\{e,\sigma\}$ and $d=d_1+d_2$. The
Grassmannian $X_d^0$ is a single point. It follows that, for each
simple reflection $s\in\W$,
$$\CR^\bullet_+(1^d,s)\IC(G/B\times X_d^0) \simeq \IC(G/B\times X_d^0)[-2]$$
where the right hand side is a complex concentrated at degree $-1$.
Clearly, $\ell(w(\bd,\sigma))=d_1d_2$. Therefore, by the above
proposition, the highest weight vector $b_{(0,0)} \in Q_\bd$ is sent
to $(-q^2)^{d_1d_2} \cdot b_{(0,0)}$ by the $R$-matrix induced from
$\CR^\bullet_+(\bd,\sigma)$.
\end{exam}

\begin{rem}
From Example \ref{exam:catR1}, Proposition \ref{prop:catR} and a compatibility
result (whose statement and proof are left to the reader) between the functor
complex $\CR^\bullet$ and the functor Res from Section \ref{sec:mod:res}, it
can be shown that the system of $R$-matrices claimed in Corollary
\ref{cor:catR} are those composed of the isomorphisms
\begin{equation}\label{eqn:R:univ}
  R_{d_1,d_2} = (-q^{\frac32})^{d_1d_2} \cdot \tau \cdot q^{\frac12H\otimes H}
  \sum_{n=0}^\infty q^{n(n-1)/2} (q-q^{-1})^n [n]_q! F^{(n)}\otimes E^{(n)}
\end{equation}
where $\tau: \Lambda_{d_1} \otimes \Lambda_{d_2} \to \Lambda_{d_2}
\otimes \Lambda_{d_1}$ is the transposition, $H$ is formally defined
by $K=q^H$.
\end{rem}

\subsection{Abelian categorification}\label{sec:R:abel}

In this subsection we translate the categorification theorem into an
abelian version.

Define algebra $A^\bullet_\bd$ as in Section \ref{sec:mod:abel}
\begin{equation}
  A^\bullet_\bd = \Ext^\bullet_{\D(G/P_\bd\times X)}(L_\bd,L_\bd)
\end{equation}
where
\begin{equation}
  L_\bd = \oplus_{S\in\CS_\bd} \IC(\overline{S}).
\end{equation}
Here $\CS_\bd$ denotes the set of the $G$-orbits of $G/P_\bd\times
X$.

\smallskip

Let $A^\bullet_\bd\mof$ denote the category of finite-dimensional
graded left $A^\bullet_\bd$-modules and let $A^\bullet_\bd\pmof$
denote the full subcategory consisting of the projectives. We apply
the same arguments as in Section \ref{sec:mod:abel}.

\begin{prop}\label{prop:tAmof}
The obvious functor $\tilde\CQ_{\bd} \to A^\bullet_\bd\pmof$ is an
equivalence of categories. Moreover, the equivalence identifies the
Grothendieck group of $A^\bullet_\bd\mof$ with $\Q(q)\otimes_\A
Q_\bd$.
\end{prop}

The main results from the previous subsection are now stated as
follows.

\begin{thm}
The followings define complexes of projective graded
$A^\bullet_{\sigma(\bd)}$-$A^\bullet_\bd$-bimodules, hence complexes
of exact functors from $A^\bullet_\bd\mof$ to
$A^\bullet_{\sigma(\bd)}\mof$
\begin{equation}
  \CR^{\bullet,\bullet}_\pm(\bd,\sigma)
  = \Ext^\bullet_{\D(G/P_{\sigma(\bd)}\times X)}(L_{\sigma(\bd)},\CR^\bullet_\pm(\bd,\sigma)L_\bd).
\end{equation}
They satisfy the homotopy equivalences and isomorphisms
\begin{align*}
  & \CR^{\bullet,\bullet}_\pm(\sigma_2(\bd),\sigma_1) \otimes
    \CR^{\bullet,\bullet}_\pm(\bd,\sigma_2) \simeq
    \CR^{\bullet,\bullet}_\pm(\bd,\sigma_1\sigma_2)
    \text{~if $\ell(\sigma_1\sigma_2) = \ell(\sigma_1) + \ell(\sigma_2)$}, \\
  & \CR^{\bullet,\bullet}_\pm(\sigma(\bd),\sigma^{-1}) \otimes
    \CR^{\bullet,\bullet}_\mp(\bd,\sigma) \simeq
    A^\bullet_\bd
    \text{~(concentrated at degree $0$)}, \\
  & \G^\bullet_{\sigma(\bd)} \otimes \CR^{\bullet,\bullet}_\pm(\bd,\sigma) \cong
    \CR^{\bullet,\bullet}_\pm(\bd,\sigma) \otimes \G^\bullet_\bd
    \text{~for $\G\in\{ \K, \K^{-1}, \E, \F \}$}.
\end{align*}
\end{thm}

\begin{cor}
Let $\K^b(A^\bullet_\bd\mof)$ be the bounded homotopic category of
$A^\bullet_\bd\mof$ and identify the Grothendieck group of
$A^\bullet_\bd\mof$ with $\Q(q)\otimes_\A Q_\bd$. Via the valuation
$$\K^b(A^\bullet_\bd\mof)\to \Q(q)\otimes_\A Q_\bd, \quad
  M^{\bullet,\bullet}\mapsto\sum_n(-1)^n[H^n(M^{\bullet,\bullet})],
$$
the system of functors
$$\CR^{\bullet,\bullet}_+(\bd,\sigma): \K^b(A^\bullet_\bd\mof) \to \K^b(A^\bullet_{\sigma(\bd)}\mof)$$
induce a system of $R$-matrices on the $U$-modules $\Q(q)\otimes_\A
Q_\bd$.
\end{cor}

\begin{rem}
The same is true if we replace $\K^b(A^\bullet_\bd\mof)$ by the
bounded derived category $\D^b(A^\bullet_\bd\mof)$.
\end{rem}


\begin{thebibliography}{9999999}

\bibitem[BBD82]{BBD82} A. Beilinson, J. Bernstein and P. Deligne,
    Faisceaux pervers, Ast\'erisque 100 (1982).

\bibitem[BGS96]{BGS96} A. Beilinson, V. Ginzburg and W. Soergel,
    Koszul duality patterns in representation theory,
    J. Amer. Math. Soc. 9 (1996), 473--527.

\bibitem[BFK99]{BFK99} J. Bernstein, I. Frenkel and M. Khovanov,
    A categorification of the Temperley-Lieb algebra and Schur quotients
    of $U(sl_2)$ via projective and Zuckerman functors,
    Selecta Math. (N.S.) 5 (1999), 199--241.

\bibitem[Bor84]{Bor84} A. Borel,
    Intersection Cohomology,
    Progress in Math., vol. 50, Birkh\"auser, Boston, 1984.

\bibitem[CK07]{CK07} S. Cautis and J. Kamnitzer,
    Knot homology via derived categories of coherent sheaves I, sl(2) case,
    arXiv:math.AG/0701194.

\bibitem[CR04]{CR04} J. Chuang and R. Rouquier,
    Derived equivalences for symmetric groups and $sl_2$-categorification,
    arXiv:math/0407205.

\bibitem[De80]{De80} P. Deligne,
    La conjecture de Weil II,
    Publ. Math. I.H.E.S. 52 (1980), 137--252.

\bibitem[Deo87]{Deo87} V. V. Deodhar,
    On some geometric aspects of Bruhat orderings II.
    The parabolic analogue of Kazhdan-Lusztig polynomials,
    J. Algebra 111 (1987), 483--506.

\bibitem[FKK98]{FKK98} I. Frenkel, M. Khovanov and A. Kirillov Jr.,
    Kazhdan-Lusztig polynomials and canonical basis,
    Transform. Groups 3 (1998), no. 4, 321--336.

\bibitem[FKS05]{FKS05} I. Frenkel, M. Khovanov and C. Stroppel,
    A categorification of finite-dimensional irreducible representations of quantum $sl(2)$ and their tensor products,
    arXiv: math.QA/0511467.

\bibitem[GMV05]{GMV05} S. Gelfand, R. MacPherson and K. Vilonen,
    Microlocal perverse sheaves,
    arXiv: math.AG/0509440.

\bibitem[KSa97]{KSa97} M. Kashiwara and Y. Saito,
    Geometric construction of crystal bases,
    Duke Math. J., 89 (1997), no. 1, 9--36.

\bibitem[KS90]{KS90} M. Kashiwara and P. Schapira,
    Sheaves on Manifolds,
    Grundlehren der math. Wiss., vol. 292, Springer, 1990.

\bibitem[Kas95]{Kas95} C. Kassel,
    Quantum Groups,
    Graduate Texts in Mathematics, vol. 155, Springer-Verlag, New York, 1995.

\bibitem[KL79]{KL79} D. Kazhdan and G. Lusztig,
    Representations of Coxeter groups and Hecke algebras,
    Invent. Math. 53 (1979), 165--184.

\bibitem[KL80]{KL80} D. Kazhdan and G. Lusztig,
    Schubert Varieties and Poincar\'e duality,
    Proc. Symp. Pure Math. 36, A.M.S., (1980), 185--203.

\bibitem[Kh00]{Kh00} M. Khovanov,
    A categorification of the Jones polynomial,
    Duke Math. J. 101 (2000), 359--426.

\bibitem[Kh03]{Kh03} M. Khovanov,
    Categorifications of the colored Jones polynomial,
    arXiv:math.QA/ 0302060.

\bibitem[KR04]{KR04} M. Khovanov and L. Rozansky,
    Matrix factorizations and link homology,
    arXiv: math.QA/0401268.

\bibitem[KMS07]{KMS07} M. Khovanov, V. Mazorchuk and C. Stroppel,
    A brief review of abelian categorifications,
    arXiv:math/0702746.

\bibitem[KW01]{KW01} R. Kiehl and R. Weisauer,
    Weil Conjectures, Perverse Sheaves and l'adic Fourier Transform,
    Springer-Verlag, Berlin, 2001.

\bibitem[Lu91]{Lu91} G. Lusztig,
    Quivers, perverse sheaves, and quantized enveloping algebras,
    J. Amer. Math. Soc. 4 (1991), no. 2, 365--421.

\bibitem[Lu93]{Lu93} G. Lusztig,
    Introduction to Quantum Groups,
    Progress in Math., vol. 110, Birkh\"auser, Boston, 1993.

\bibitem[Ma03]{Ma03} A. Malkin,
    Tensor product varieties and crystals: the $ADE$ case.
    Duke Math. J. 116 (2003), no. 3, 477--524.

\bibitem[Na94]{Na94} H. Nakajima,
    Instantons on ALE spaces, quiver varieties, and Kac-Moody algebras,
    Duke Math. J. 76 (1994), no. 2, 365--416.

\bibitem[Na01]{Na01} H. Nakajima,
    Quiver varieties and tensor products.
    Invent. Math. 146 (2001), no .2, 399--449.

\bibitem[Ras03]{Ras03} J. A. Rasmussen,
    Khovanov homology and the slice genus,
    arXiv:math.GT/0306378.

\bibitem[RT90]{RT90} N. Yu. Reshetikhin, V. G. Turaev,
    Ribbon graphs and their invariants derived from quantum groups,
    Comm. Math. Phys. 127 (1990), 1--26.

\bibitem[Str05]{Str05} C. Stroppel,
    Categorification of the Temperley-Lieb category, tangles, and cobordisms via projective functors,
    Duke Math. J. 126(3), (2005), 547--596.

\bibitem[Wa04]{Wa04} I. Waschkies,
    The stack of microlocal perverse sheaves,
    Bull. Soc. Math. France 132 (2004), 397--462.

\end{thebibliography}
\end{document}